\documentclass[a4paper,reqno, 11pt]{amsart}
\usepackage{graphicx,color}
\usepackage{footmisc,amsthm}
\usepackage{amsmath}
\usepackage{amsfonts}
\usepackage{amssymb}
\usepackage{bbm}
\usepackage{epstopdf}
\usepackage{color}
\numberwithin{equation}{section}
\usepackage{tikz}
\usepackage{pgfplots}
\usepackage{subfig}

\newtheorem{Theorem}{Theorem}

\newtheorem{Remark}[Theorem]{Remark}

\newtheorem{Assumption}[Theorem]{Assumption}
\newtheorem{Claim}[Theorem]{Claim}
\newtheorem{Lemma}[Theorem]{Lemma}
\newtheorem{Definition}[Theorem]{Definition}
\newtheorem{Proposition}[Theorem]{Proposition}

\numberwithin{Theorem}{section}

\title[Hydrodynamic limit for tiling dynamics]{Lozenge tiling dynamics and convergence to the hydrodynamic equation}
\author{Beno\^it  Laslier}
\address{LPMA - Univ. Paris Diderot,
B\^atiment Sophie Germain,
avenue de France,
75013 Paris, France
\\
E-mail: laslier@math.univ-paris-diderot.fr}
\author{Fabio Lucio Toninelli}
\address{Univ Lyon, CNRS, Universit\'e Claude Bernard Lyon 1,  UMR 5208, Institut Camille Jordan, F-69622 Villeurbanne cedex, France
 \\
E-mail: toninelli@math.univ-lyon1.fr}

\begin{document}

\maketitle

\begin{abstract}
  We study a reversible continuous-time Markov dynamics of a discrete
  $(2+1)$-dimensional interface. This can be alternatively viewed as a
  dynamics of lozenge tilings of the $L\times L$ torus, or as a conservative dynamics for a 
  two-dimensional system of interlaced particles. The particle
  interlacement constraints imply that the equilibrium measures are
  far from being product Bernoulli: particle correlations decay like the
  inverse distance squared and interface height fluctuations behave on large scales like a massless
  Gaussian field.  We consider a particular choice of the transition
  rates, originally proposed in \cite{LRS}: in terms of interlaced
  particles, a particle jump of length $n$ that preserves the
  interlacement constraints has rate $1/(2n)$.  This dynamics presents
  special features: the average mutual volume between two interface
  configurations decreases with time \cite{LRS} and a certain one-dimensional
  projection of the dynamics is described by the heat equation
  \cite{Wilson}.

  In this work we prove a hydrodynamic limit: after a diffusive
  rescaling of time and space, the height function evolution tends as
  $L\to\infty$ to the solution of a non-linear parabolic PDE. The initial profile is assumed to be $C^2$ differentiable and to contain no ``frozen region''. The
  explicit form of the PDE was recently conjectured \cite{LThydro} on
  the basis of local equilibrium considerations.  In contrast with the hydrodynamic equation for the
  Langevin dynamics of the Ginzburg-Landau model \cite{FunakiSpohn,Nishikawa},
  here the mobility coefficient turns out to be a non-trivial function
  of the interface slope.
\end{abstract}


\section{Introduction}

This work is motivated by the problem of understanding the large-scale
limit of stochastic interface evolution \cite{SpohnJSP,Funaki}. More
precisely we are interested in the motion of the interface between
two stable thermodynamic phases in a $(d+1)$-dimensional
medium. Needless to say, the physically most interesting case is that
of $d=2$.  The motion results from thermal fluctuations; overall
the interface tends to  flatten, thereby minimizing its free energy.

Mathematically we consider an \emph{effective interface approximation}
where the internal structure of the two phases above and below the
interface is disregarded and the $d$-dimensional interface is modeled
as the graph of a function from $\mathbb R^d$ to $\mathbb R$ (or some discretized
version of this).  The physical rationale behind this approximation is a time-scale separation:
 the internal degrees of freedom of the two phases relax much faster than those of the
interface. The effect of thermal fluctuations is modeled as a
Markov chain and the fact that phases are in coexistence translates to
reversibility of that chain.


On phenomenological grounds \cite{SpohnJSP} one expects that, on sufficiently coarse scales, the interface dynamics is
deterministic and described by a hydrodynamic equation of the type
\begin{eqnarray}
  \label{eq:eqgenerale}
  \partial_t \psi(x,t)=-\mu(\nabla \psi(x,t))\frac{\delta F[\psi]}{\delta \psi(x,t)},
\end{eqnarray}
where $F[\psi]$ is the equilibrium surface free energy
 and $\mu(\nabla\psi)$ is a
mobility coefficient, that depends on the details of the dynamics. Note that, without the prefactor $\mu$, \eqref{eq:eqgenerale} 
would be simply the gradient flow associated
with the surface energy functional. We can therefore interpret $\mu$ as describing how effective  the relaxation produced by the dynamics is, hence the name ``mobility''.
The relevant space-time rescaling where the  behavior described by \eqref{eq:eqgenerale} should emerge is the diffusive one.

 Writing $F$ as the integral of the
slope-dependent surface tension,
\begin{eqnarray}
  F[\psi]=\int \mathrm{d}x_1 \mathrm{d}x_2\, \sigma(\nabla\psi),
\end{eqnarray}
the equation takes the parabolic form
 \begin{eqnarray}
   \label{eq:PDEgen}
   \partial_t \psi=\mu(\nabla \psi)\sum_{i,j=1,2}\sigma_{i,j}(\nabla \psi)\frac{\partial^2}{\partial_{x_i}\partial_{x_j}}\psi
 \end{eqnarray}
 where parabolicity derives from convexity of $\sigma$. In general,
 the PDE \eqref{eq:PDEgen} is non-linear and the mobility is expected
 to be a non-trivial  function of the slope.

It is very difficult, in particular if $d>1$, to prove convergence to
a hydrodynamic equation starting from a non-trivial microscopic
model. Even worse, in general it is not possible to guess, even
heuristically, an explicit expression for the mobility; in fact, its expression as provided by the Green-Kubo formula involves an integral of
space-time correlations computed in the stationary states \cite{SpohnJSP}. Exceptions where
$\mu$ can be written down explicitly are usually models where the
dynamics satisfies some form of ``gradient condition'', i.e. the
microscopic current is the lattice gradient of some function
\cite{Spohn}.  The only example we know of a rigorous proof
of the hydrodynamic limit for a $d>1$ diffusive interface dynamics is
the work \cite{FunakiSpohn} by Funaki and Spohn. There, the Langevin dynamics for the
$d\ge 1$ Ginzburg-Landau model with symmetric and convex potential is
studied and the convergence to a hydrodynamic limit of the type \eqref{eq:PDEgen} is
proven. See also \cite{Nishikawa} where the analogous result is proven in a domain with Dirichlet instead of periodic boundary conditions. It is important to remark that for the Ginzburg-Landau model  the mobility turns
out to be a constant, that can be set to $1$ by a trivial time-change.

In dimension $d=1$, instead, natural Markov dynamics of discrete interfaces are
provided by conservative lattice gases on $\mathbb Z$ (e.g. symmetric exclusion
processes or zero-range processes), just by interpreting the number of
particles at site $x$ as the interface gradient $\phi(x)-\phi(x-1)$ at
$x$. Similarly, conservative continuous spin models on $\mathbb Z$ translate into Markov dynamics for one-dimensional interface models with continuous heights. Then, a hydrodynamic limit for the height function $\phi$ follows
from that for the particle density (see e.g. \cite[Ch. 4 and 5]{KL}
for the symmetric simple exclusion and for a class of zero-range
processes, and for instance \cite{Fritz} for the $d=1$ Ginzburg-Landau model). For $d>1$, instead, there is in general no natural way of
associating a height function to a particle system on $\mathbb Z^d$.

\medskip

In the present work, we study a two-dimensional stochastic interface
evolution for which we can obtain a hydrodynamic limit of the type
\eqref{eq:PDEgen}. Our model is very different from the
Ginzburg-Landau one. First of all, the interface is discrete (heights
take integer values) so that the dynamics is a Markov jump process
rather than a diffusion.  More importantly, the mobility coefficient
$\mu$ in the limit PDE is a non-constant (and actually non-linear)
function of the interface slope.

The state space of our Markov chain is the set of lozenge tilings of
the two-dimensional triangular lattice, or more precisely of its
$L\times L$  periodization.  Lozenge tilings of the triangular lattice are
well known to be in bijection with perfect matchings (aka dimer
coverings) of the dual lattice, which is the honeycomb lattice
$\mathcal H$. See Fig. \ref{fig:2n}.\begin{figure}[h]
  \includegraphics[width=7cm]{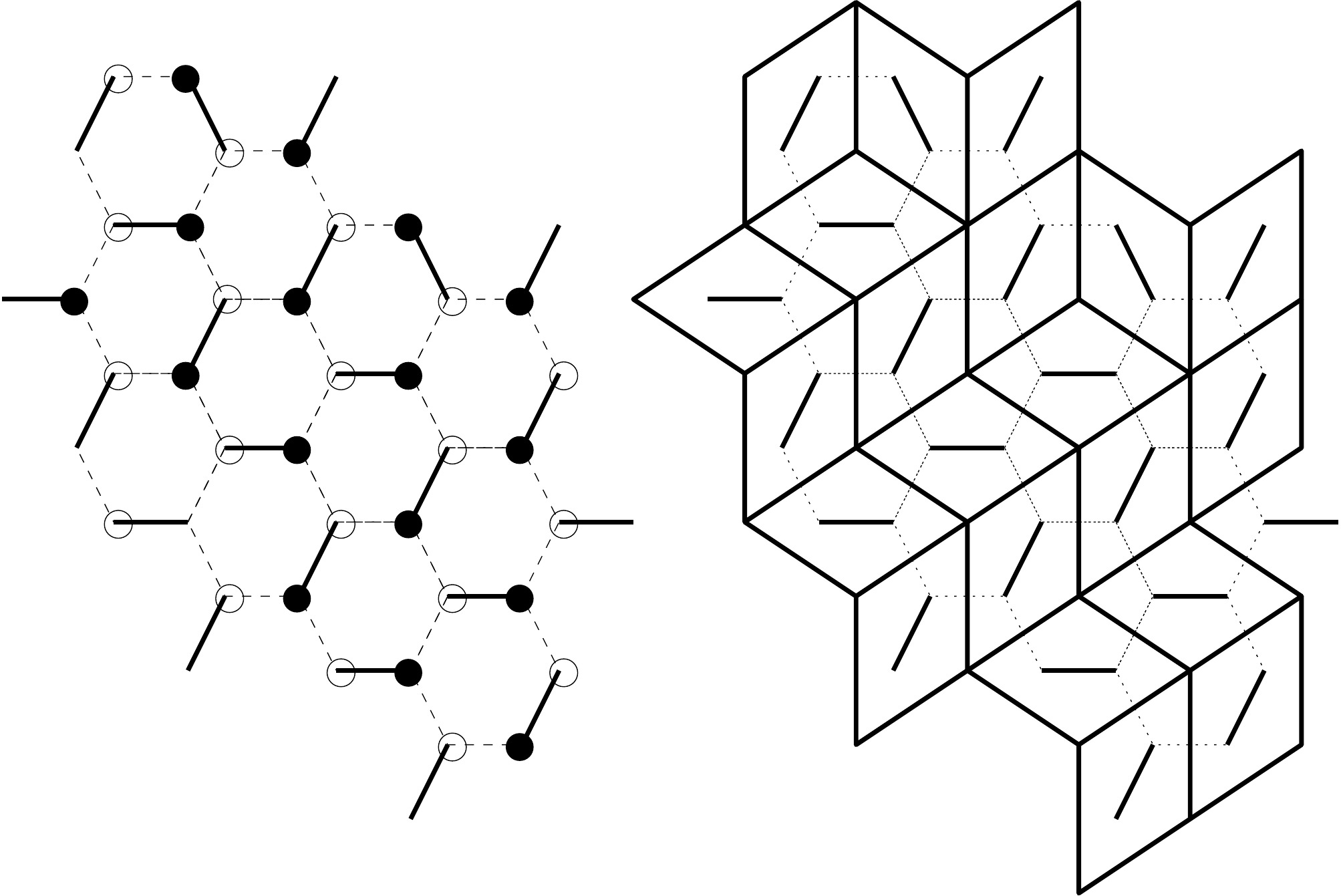}
  \caption{The bijection between dimer coverings of $\mathcal H$ (left) and lozenge tilings of the triangular lattice (right).  The horizontal dimers will be called ``particles''. }
\label{fig:2n}
\end{figure}
On the other hand, since $\mathcal H$ is planar and
bipartite, its perfect matchings are  in bijection with a
height function defined on its faces, see Fig. \ref{fig:monsurf}.  This height function will then
be our model of the discrete two-dimensional interface. 
\begin{figure}[h]
  \includegraphics[width=4cm]{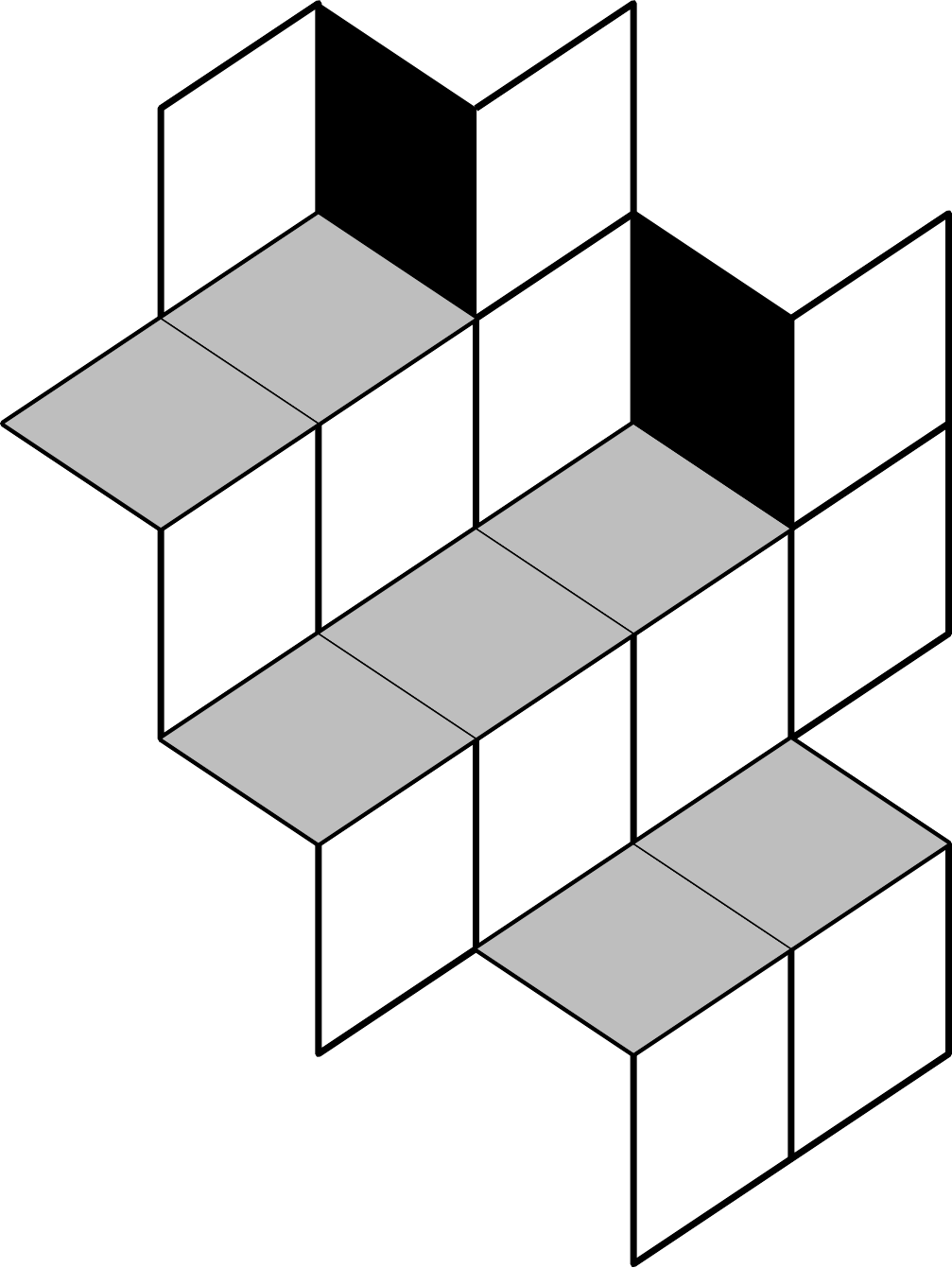}
  \caption{The lozenge tiling of Fig. \ref{fig:2n},  seen as a stepped two-dimensional surface representing the boundary of 
a pile of unit three-dimensional cubes. }
\label{fig:monsurf}
\end{figure}
 The dynamics we
study is reversible with respect to a two-parameter family of ergodic Gibbs measures, that are locally uniform measures 
on lozenge tilings. These measures are labelled by the two parameters of the interface slope  and have a determinantal representation \cite{KOS}. 
It is important to recall that such measures are far from looking like product Bernoulli measures: indeed, correlations decay like the inverse distance squared, while the height function itself tends in the large-scale limit to a two-dimensional
massless Gaussian field \cite{KOS,Kenyon}.

Starting with Section \ref{sec:model}  we will adopt the dimer instead of the tiling
point of view, and actually we will view the dimer configuration as an
interacting particle system (``particles'' being the horizontal dimers as in Figure \ref{fig:2n}). Such  particles satisfy particular interlacement
conditions, recalled in Section \ref{sec:model}.

\medskip

The updates of the dynamics  consist in a horizontal dimer
(or particle) jumping a certain distance $n$ vertically (up or down)
in the hexagonal lattice, and such a transition will be assigned a rate
$1/(2n)$ (the prefactor $1/2$ is there just to conform with the
previous literature). See Section \ref{sec:model} for a precise
definition. The jumps are not always allowed, not only because
particles cannot superpose, but also because jumps cannot violate the
above-mentioned interlacement constraints. This point is very
important: in fact, it is the interlacement constraints that are
responsible for the long-range correlations in the equilibrium
measure.  If instead we had only the exclusion constraint, the
equilibrium measure would be Bernoulli.

Let us call ``long-jump dynamics'' the lozenge tiling dynamics just
described, to distinguish it from the ``single-flip Glauber dynamics''
where only jumps of length $|n|=1$ are allowed. Let us mention that
the single-flip dynamics is equivalent to the Glauber dynamics of the
three-dimensional Ising model at zero temperature with Dobrushin
boundary conditions \cite{CMT}.  The long-jump dynamics has an
interesting story. It was originally introduced in \cite{LRS} with the
goal of providing a Markov chain that approaches the uniform measure
on tilings, in variation distance, in a time that is polynomial in the
system size $L$. In fact, the key point is that the mutual volume
between two interface configurations is a super-martingale which,
together with attractiveness of the dynamics, allows one to deduce
polynomial mixing via coupling arguments. Later, in \cite{Wilson} it
was proven that the total variation mixing time of the long-jump
dynamics is actually of order $L^2\log L$, and that (in special
domains) a particular one-dimensional projection of the height
function evolves according to the one-dimensional heat equation. The
results of \cite{LRS,Wilson} were used as a building block in
\cite{CMT,LTcmp} to prove that, under some conditions on the geometry
of the domain, the mixing time of the single-flip Glauber dynamics is
of order $L^{2+o(1)}$, like that of the long-jump dynamics. Finally,
in \cite{LThydro} we discovered that, in contrast with the single-flip
dynamics, the long-jump one satisfies certain identities, that allow
to conjecture an explicit form for the hydrodynamic equation,
cf. \eqref{eq:PDE} or equivalently \eqref{eq:PDE2}.  While the
dynamics does not satisfy a gradient condition in the usual sense that
the microscopic current is a discrete gradient, still a discrete
summation by parts causes the ``dangerous'' part in the mobility
coefficient, the one involving space-time correlations, to vanish
\cite[Sec. 4]{LThydro}.  A reflection of this fact is the summation by
parts that takes place in Proposition \ref{prop:id} below.

\smallskip

In the present
work, we prove rigorously the convergence to the 
hydrodynamic equation, in the case of periodic boundary
conditions, under suitable smoothness assumptions on the initial profile. See Theorem \ref{th:main}.

\medskip

The strong correlations in the invariant measures seem to prevent an
approach to the hydrodynamic limit problem based on the application of
classical methods going through one- and two-block estimates
\cite{Spohn,KL}.  Instead, broadly speaking, the proof of our result
follows a scheme that is similar to that of \cite{FunakiSpohn} for the
Ginzburg-Landau model, that builds on an extension of the so-called
$H^{-1}$ (norm) method first introduced by Chang and Yau \cite{CY}.  The
basic idea of the method is to prove that the time-derivative of the
${\mathbb L}^2$ distance between the solution of the PDE and the randomly
evolving interface is non-positive in the infinite volume limit.
There are however important differences between the application of the
method to the Ginzburg-Landau model and to ours, and here we mention
two of them (see also Remark \ref{rem:noRiesz}). First of all, one of
the key points in \cite{FunakiSpohn} is an a-priori ${\mathbb L}^2$ control of
interface gradients out of equilibrium, see Lemma 4.1 there, which is
based on a simple coupling argument, that works because the
interaction potential is assumed to be strictly convex. In our case
the analog would be an a-priori control of the variable we call $k$ in
\eqref{eq:ka}, that determines how far particles can jump. The
coupling argument however does not work, probably because the
interaction in our model \emph{is not} strictly convex, as witnessed
by the fact that the position of a particle given its neighbors is
uniformly distributed in the equilibrium state. Therefore, we  have
to proceed differently to get tightness of $k$, see e.g. Proposition
\ref{prop:contrk}.  Secondly, the fact that the mobility is constant
for the Ginzburg-Landau model played an important role in
\cite{FunakiSpohn} and was behind the fact that the time derivative of
the $\mathbb L^2$ distance between the randomly evolving interface and the
solution of the deterministic hydrodynamic equation is negative. In
fact, the negativity of this derivative is basically a consequence of
convexity of the surface tension, see \cite[Eq. (5.7)]{FunakiSpohn}
and subsequent discussion.  Contraction of the $\mathbb L^2$ norm holds also
for our model but it seems to be more subtle, as it requires
the negativity of a certain function (cf. \eqref{eq:Gg}), a fact that 
\emph{does not} follow simply from a thermodynamic convexity.

Our result raises interesting questions that we leave for future
research: notably, the study of space-time correlations of height
fluctuations around the limit PDE and the proof of the convergence to
the hydrodynamic limit when boundary conditions are not periodic,
especially when boundary conditions are such as to impose frozen
regions in the equilibrium macroscopic shape \cite{Kenyon}.

\subsection*{Organization of the article}
In Section \ref{sec:model}, we introduce more precisely our model and
we state our main result, Theorem \ref{th:main}: the convergence of the height function to
the solution of an explicit deterministic PDE. This equation is of parabolic, non-linear type and in
Section \ref{sec:PDE} we deduce existence and regularity of its
solution (i.e. statement (i) of Theorem \ref{th:main}) from known
results on parabolic PDEs. Sections \ref{sec:DLt} to \ref{sec:calcolo}, that are the bulk
of the work, contain the proof of item (ii) in Theorem \ref{th:main}, i.e. the hydrodynamic limit convergence statement itself. A
short reader's guide to the structure of the proof is given at the end
of Section \ref{sec:model}.

\section{Model and results}
\label{sec:model}
\subsection{Dimer configurations and height function}
Let $\mathcal H$ denote the infinite planar  honeycomb lattice and let
$\hat e_1,\hat e_2,\hat e_3$ be the unit vectors depicted in Figure \ref{fig:interlac}.  The
dual lattice $\mathcal T=\mathcal H^*$ is a triangular lattice.  We let
$\mathcal H_L$ be the periodization, with period $L$, of $\mathcal H$
in directions $\hat e_1,\hat e_2$. The dual graph of $\mathcal H_L$, denoted 
$\mathcal T_L$, is a periodized triangular lattice.

\begin{figure}[h]
  \includegraphics[width=9cm]{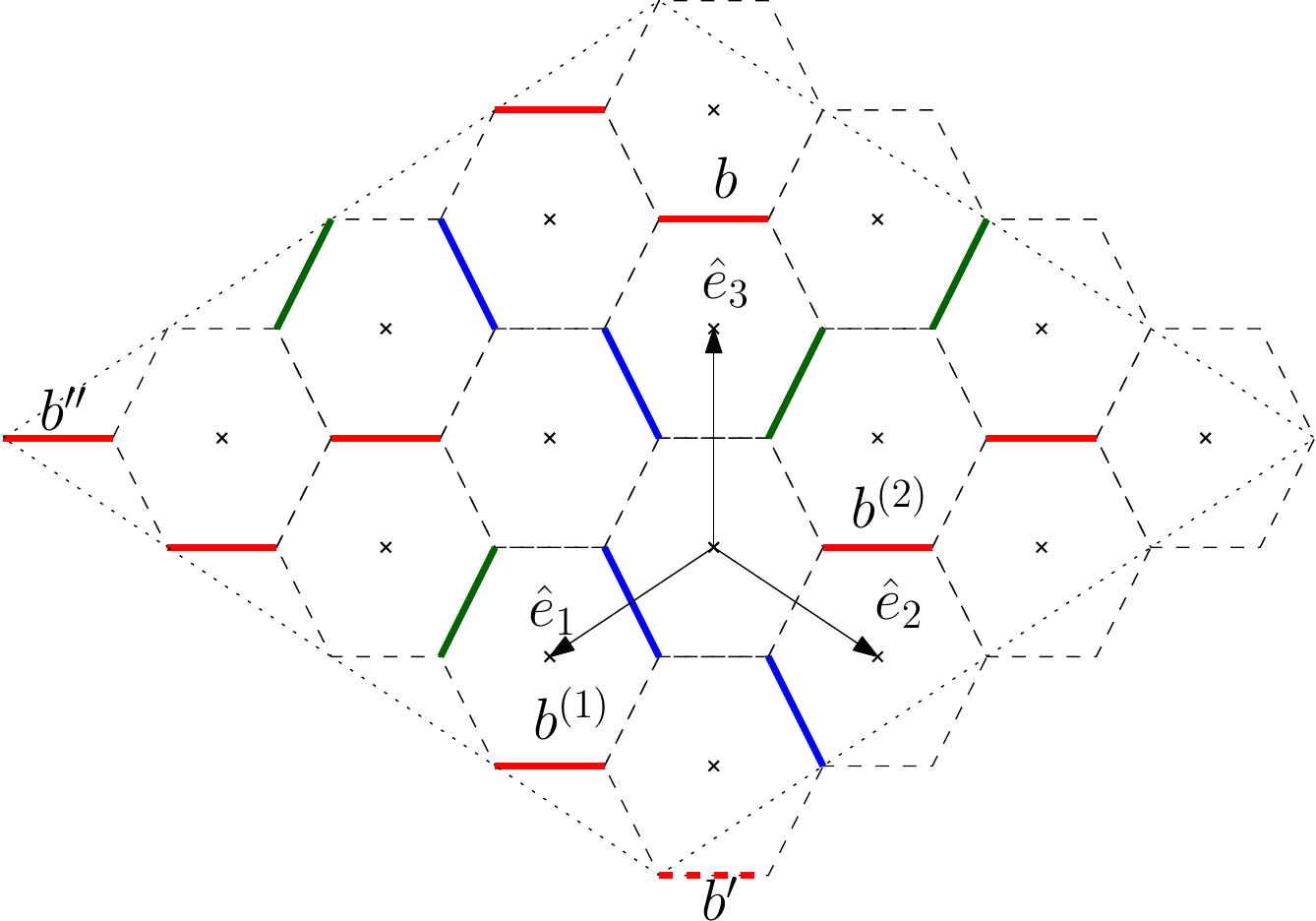}
  \caption{The hexagonal lattice $\mathcal H$ with the unit vectors
    $\hat e_i,i=1,2,3$.  The portion of $\mathcal H$ contained in the
    dotted region is the periodized lattice $\mathcal
    H_L,L=4$. Crosses are vertices of the periodized triangular dual
    lattice $\mathcal
    T_L$. 
    of $\mathcal H_L$.
    Thick edges are dimers of a perfect matching of $\mathcal H_L$
    (the matching can be extended to a $4$-periodic perfect matching
    of the infinite lattice $\mathcal H$).  Dimers are interlaced: for
    instance, type-3 (i.e., horizontal) dimers $b,b'$ are interlaced
    with type-3 dimers $b^{(1)}, b^{(2)}$. (Note that dimer $b'$ is
    the same (by periodicity) as dimer $b''$; that is why it is drawn
    as dashed). In this drawing there are exactly $4$ dimers of types
    $1$ and $2$ and $8$ dimers of type $3$.
    }
\label{fig:interlac}
\end{figure}

On $\mathcal T_L$ we introduce coordinates $u=(u_1,u_2)$ with
$0\le u_i\le L-1$ by assigning coordinates $(0,0)$ to an arbitrarily
fixed vertex of $\mathcal T_L$ (the ``origin'') and letting
$\hat e_1,\hat e_2$ have coordinates $(1,0)$ and $(0,1)$
respectively. Note that the ``unit vector'' $\hat e_3$ has then
coordinates $\hat e_3=(-1,-1)$.

The bipartite graph $\mathcal H_L$ has $2L^2$ vertices and every perfect matching $\eta$ of $\mathcal H_L$ contains exactly $L^2$ edges. Edges in a  perfect matching $\eta$ are called dimers and $\eta$ is referred to as a dimer covering.  We will say that an edge (or a dimer) is of type $i=1,2,3$ if it is
parallel to the edge crossed by $\hat e_i$. The set $\Omega_L$ of all perfect
matchings of $\mathcal H_L$ can be decomposed according to the number
of dimers of the three types, as follows.  Given
$\bar\rho^{(L)}=(\bar\rho^{(L)}_1,\bar\rho^{(L)}_2)$ with
$\bar\rho^{(L)}_i\in(0,1)$, $\sum_{i=1}^2 \bar\rho_i^{(L)}\in(0,1)$
and $L\bar\rho^{(L)}_i\in\mathbb N$, let $\Omega_{\bar\rho^{(L)}}$ be
the set of perfect matchings of $\mathcal H_L$ with
$L^2 \bar\rho^{(L)}_i$ dimers of type $i=1,2$. Note that the number of
dimers of type $3$ is then constrained to be
$L^2 \bar\rho^{(L)}_3:=L^2(1-\bar\rho^{(L)}_1-\bar\rho^{(L)}_2)$.

Let $\tau_u$ denote the translation by $u=(u_1,u_2)$.
We recall a few known (and easy to verify) facts:
\begin{enumerate}
\item Given the locations of all dimers of a given type $i\in\{1,2,3\}$, the whole dimer covering $\eta$ is uniquely determined. In the following, we will call dimers of type $3$ (i.e. the horizontal ones) ``particles''. 

\item
The positions of particles in a vertical column of hexagons are interlaced with those in the two neighboring columns. More explicitly,
if there is a particle at a horizontal edge
  $b$ and another  at the edge
  $b'=\tau_{-n \hat e_3} b$ for some $n>0$ (note that $b,b'$ are  in the same column), then 
  for $j=1,2$ there  is a particle at an 
  edge $b^{(j)}=\tau_{\hat e_j-r_j \hat e_3}b$, for some
  $r_j\in\{0,\dots, n-1\}$.  
  See
  Figure \ref{fig:interlac} and \ref{fig:Ip}. Similar interlacing conditions hold for dimers of types $1$ and $2$. 
\item Every cycle of length $L$ on $\mathcal T_L$ in direction
  $\hat e_i$ crosses the same number of dimers of type $i$ (and
  therefore $L \bar\rho^{(L)}_i$ of them). This is a consequence of
  property (2).

\end{enumerate}

We introduce a height function $H_\eta(\cdot)$ (defined on $\mathcal T_L$, i.e. on hexagonal faces of $\mathcal H_L$) for configurations  $\eta\in\Omega_{\bar\rho^{(L)}}$.
We first need some notation:
\begin{Definition}
\label{def:Ai}
Given $u\in \mathcal T_L$ and $i=1,2,3$, let $b_i(u)$ be the edge
(necessarily of type $i$) of $\mathcal H_L$ that is crossed by the
edge of $\mathcal T_L$ that joins $u$ to $u+\hat e_i$.
\end{Definition}
 Then, we establish:
 \begin{Definition}[Height function]
\label{def:height}
 For every $\eta\in \Omega_{\bar\rho^{(L)}}$ we set $H_\eta(0,0)=0$ and
\begin{eqnarray}
  \label{eq:gradh}
  H_\eta(u+\hat e_i)-H_\eta(u)=-\bar\rho^{(L)}_i+{\bf 1}_{b_i(u)\in \eta},\quad i=1,2,3.
\end{eqnarray}
 \end{Definition}
 The height function is well defined, i.e. the sum of the gradients
 along any closed cycle is zero. Indeed, when the cycle has trivial
 winding numbers around the torus $\mathcal T_L$ it suffices to verify
 that the total height change is zero along any elementary cycle of
 the type
 \[u\to u+\hat e_1\to u+\hat e_1+\hat e_2\to u+\hat e_1+\hat e_2+\hat e_3=u.\]
In this case, the height change is
\[
-\sum_{i=1,2,3}\bar\rho^{(L)}_i+{\bf 1}_{b_1(u)\in \eta}+{\bf 1}_{b_2(u+\hat e_1)\in\eta}+{\bf 1}_{b_3(u+\hat e_1+\hat e_2)\in\eta}=0
\]
because exactly one of the three crossed edges, that share a common vertex, is occupied by a dimer
and $\sum_{i=1,2,3} \bar\rho^{(L)}_i=1$. When instead the cycle winds once around
$\mathcal T_L$ following direction $\hat e_i$, then the height change
is zero because exactly $L\bar\rho^{(L)}_i$ dimers of type $i$ are
crossed (property (3) above).   The case of a general cycle easily follows.

Let $\mathbb T\subset \mathbb R^2$ be the \emph{open} triangle of vertices
$(0,0),(1,0),(0,1)$, so that $\bar\rho^{(L)}\in\mathbb T$.  Assume
that the sequence $\{\bar\rho^{(L)}\}_{L\in\mathbb N}$ converges to
some $\bar\rho\in \mathbb T$ and that a sequence
$\{\eta_0^{(L)}\}_{L\in\mathbb N}$ of configurations in
$\Omega_{\bar\rho^{(L)}}$ approximates a smooth periodic function
$\psi$, in the sense that
\begin{eqnarray}
  \label{eq:iniziale}
  \lim_{L\to\infty} \frac1L H_{\eta_0^{(L)}}(\lfloor u L\rfloor)=\psi(u), \; \text{for every } u\in [0,1]^2.
\end{eqnarray}
Then, from the definition of height function we see that
\begin{eqnarray}
  \label{eq:iniz2}
\nabla\psi(u)+(\bar \rho_1,\bar\rho_2)\in \mathbb T\cup\partial \mathbb T,\; \text{for every } u\in [0,1]^2.  
\end{eqnarray}

\subsection{The dynamics}
\label{sec:thedyn}
We consider a Markov dynamics on  $\Omega_{\bar\rho^{(L)}}$, that was introduced in \cite{LRS} and later studied in \cite{Wilson} (in both these references, the dynamics is defined on the set of dimer coverings of a planar, rather than periodized, subset of $\mathcal H$). Also, with respect to \cite{LRS,Wilson} we multiply transition rates by $L^2$, in order to avoid rescaling time in the hydrodynamic limit.

Recall from the previous section that the configuration $\eta$ is fully determined by the positions of the $L^2\bar\rho^{(L)}_3$ dimers of type $3$, or particles. The dynamics will be then defined directly in terms of particle 
moves.

Assign a label $p=1,\dots, L^2\bar\rho^{(L)}_3$ to particles. Given a
configuration $\eta$ and a particle label $p$, let $I^+(p)\ge0$
(resp. $I^-(p)\ge0$) be the maximal displacement particle $p$ can take
in direction $+\hat e_3$ (resp. $-\hat e_3$), without violating the interlacement constraints, when all other particles
are kept fixed, cf.
Fig. \ref{fig:Ip}.
\begin{figure}[h]
  \includegraphics[width=2cm]{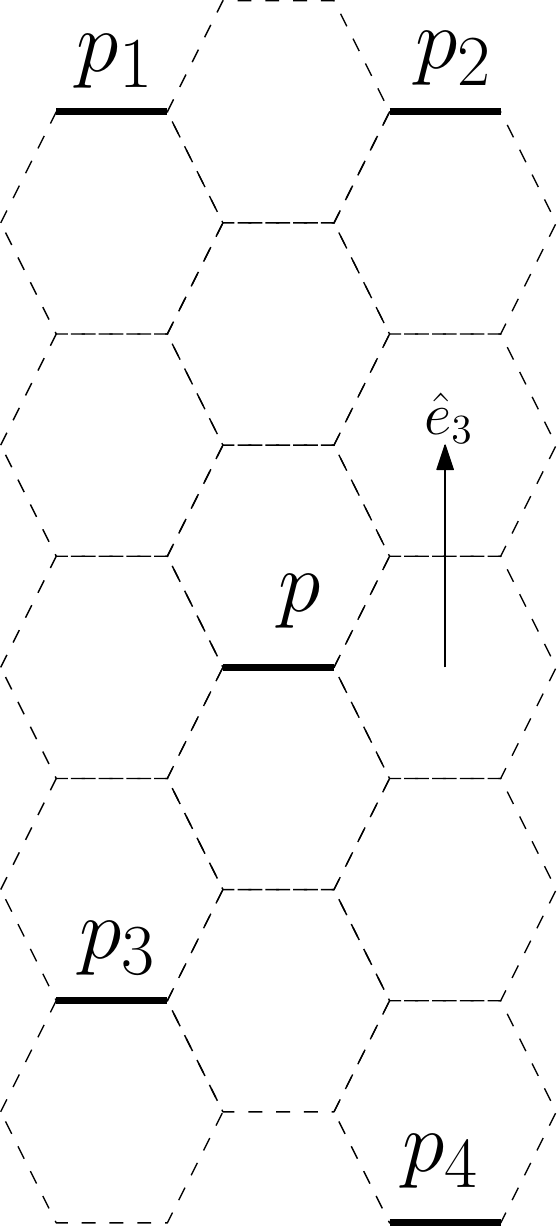}
  \caption{If particles $p_1,\dots,p_4$ are kept fixed, then particle $p$ can be moved in direction $+\hat e_3$ (resp. $-\hat e_3$) by at most $2$ steps (resp. $1$ step). Therefore, $I^+(p)=2, I^-(p)=1$.  }
\label{fig:Ip}
\end{figure}

\begin{Definition}[Transition rates]
The possible updates of the dynamics are the following:
\begin{itemize}
\item a particle $p$ moves by $ n \hat e_3, 0<n\le I^+(p)$. This transition has rate $L^2/(2n)$.
\item a particle $p$ moves by $ -n \hat e_3, 0<n\le I^-(p)$. This transition has
also  rate $L^2/(2n)$.

\end{itemize}
 
\end{Definition}
Observe  that as soon as
there are at least two particles in each cycle in direction $\hat e_3$
(which is certainly the case in our setting for $L$ large, since this number is
$L \bar\rho^{(L)}_3$ and we  assume  that
$\rho^{(L)}_3$ tends to a positive constant),  a particle $p$
cannot reach the same position via a jump by
$ +n \hat e_3, 0<n\le I^+(p)$ as with a jump by
$ -m \hat e_3, 0<m\le I^-(p)$.  Therefore, even if we are on the torus
it makes sense to distinguish between ``upward'' and ``downward''
particle jumps.

The configuration at time $t$ will be denoted by $\eta(t)$ and the law
of the process by $\mathbb P$.  The uniform measure $\pi_L$ on
$\Omega_{\bar\rho^{(L)}}$ is reversible (because transition rates are
symmetric) and actually it is the unique stationary measure, since the
Markov chain is ergodic (that any configuration in
$\Omega_{\bar\rho^{(L)}}$ can be reached from any other via particle
jumps of size $1$ is a classical argument \cite{LRS} based on the
definition of height function; for a proof in the periodized setting,
see also \cite[Lemma 1]{CT}).

We will formulate a hydrodynamic limit result for the height function.
However, recall that only height increments are really defined by the
particle configuration and one has to make a choice of an arbitrary
global additive constant. In the definition of $H_\eta$ above we fixed
such constant so that the height is zero at the origin.  In the
dynamical setting the natural choice is to fix the height at the
origin to be the ``integrated current at time $t$'' (which is zero at
initial time):
\begin{Definition}[Integrated current]
\label{def:J}
For $u\in \mathcal T_L$ and $t>0$ we define the integrated current $J(u,t)$ as the number of particles that cross $u$ downward minus the number of particles that cross $u$ upward in the time interval $[0,t]$.
Also, we let
\begin{eqnarray}
\label{latamm}
  H(u,t)=H_{\eta_0}(u)+J(u,t).
\end{eqnarray}  
\end{Definition}
\begin{Remark}
\label{rem:1}
Note  that $H(\cdot,t)$ is just the height function $H_{\eta(t)}$ (in the sense of Definition \ref{def:height}) of the configuration at time $t$, up to the global constant $H(0,t)=J(0,t)$.
In particular, the analog of \eqref{eq:gradh} holds at time $t$:
\begin{eqnarray}
  \label{eq:gradh2}
  H(u+\hat e_i,t)-H(u,t)=-\bar\rho^{(L)}_i+{\bf 1}_{b_i(u)\in \eta(t)}.
\end{eqnarray}
\end{Remark}

To understand the rationale behind Definition \ref{def:J} note first of all that, in terms
of dimer covering, moving a particle by $\pm \hat e_3$ corresponds to
performing an elementary rotation of three dimers around a hexagonal
face $u$ of $\mathcal H_L$. Say that $u$ is not the origin, where $H_\eta$ is fixed to zero. Then, according to Definition \ref{def:height}, the height at $u$ changes by $\mp1$ as an effect of the rotation. Similarly, a move by $\pm m\hat e_3$
corresponds to the concatenation of $m$ rotations around $m$ adjacent,
vertically stacked, hexagonal faces. See Figure
\ref{fig:rotazioni}. Correspondingly, when particle $p$ jumps by
$\pm n\hat e_3$, the function $H(u,t)$ defined as in \eqref{latamm} changes by
$\mp 1$ at the $n$ positions crossed by the particle.

\begin{figure}[h]
  \includegraphics[width=7cm]{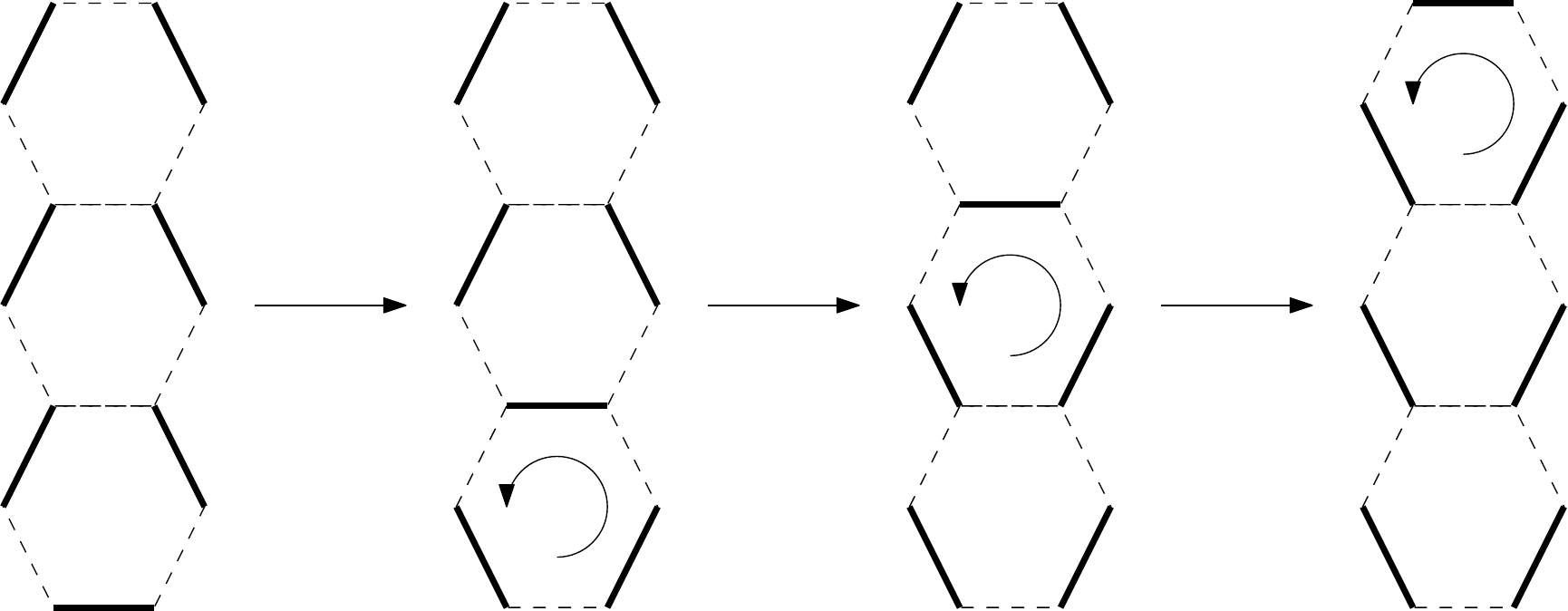}
\caption{A particle jump of length $n$ ($n=3$ here) corresponds to $n$ elementary dimer rotations in $n$ vertically stacked
hexagonal faces.}
\label{fig:rotazioni}
\end{figure}

\subsection{Hydrodynamic limit}
We assume that the initial condition of the dynamics approximates a smooth profile, in the sense of  \eqref{eq:iniziale}, but we impose a  condition stronger than \eqref{eq:iniz2} on the gradient:
\begin{Assumption}
\label{ass:eta0}
Let $\bar\rho=(\bar\rho_1,\bar\rho_2)\in \mathbb T$ be fixed. The initial condition $\eta_0=\eta_0^{(L)}$ belongs to $\Omega_{\bar\rho^{(L)}}$ and $\bar\rho^{(L)}\to\bar\rho$ as $L\to\infty$. 
Moreover, there exists a compact, convex subset $\mathcal A\subset \mathbb T$ and a periodic, $C^2$ function $\psi_0$ on the torus   $[0,1]^2$ satisfying $\psi_0(0,0)=0$ and 
\begin{eqnarray}
  \label{eq:etazero}
\nabla \psi_0(u)+\bar\rho\in \mathcal A,   \quad\forall u\in[0,1]^2
\end{eqnarray}
such that \eqref{eq:iniziale} holds.
\end{Assumption}
Our goal is to prove a hydrodynamic limit (in the diffusive scaling) for the height function $H$, under Assumption \ref{ass:eta0}.
See Remark \ref{rem:regolarita} below for a discussion of whether such an assumption can be relaxed.

First, we define a function $W=(W_1,W_2):\mathbb T\mapsto \mathbb R^2$ as
\begin{equation}
  \begin{array}{ll}
  W_1(x,y)=-\frac1{2\pi}\cot(\pi(x+y))\sin^2(\pi y)-\frac y4+\frac1{4\pi}\sin(2\pi y)\\
W_2(x,y)=W_1(y,x).   
  \end{array}
 \end{equation}

Our main result is the following:
\begin{Theorem}
\label{th:main}
  Let $\eta_0=\{\eta_0^{(L)}\}_L$ satisfy Assumption \ref{ass:eta0}. Then:
  \begin{enumerate}
  \item [(i)] There exists 
  a unique smooth solution  of the Cauchy problem
   \begin{eqnarray}
     \label{eq:PDE}
\left\{     \begin{array}{l}
       \partial_t \psi(u,t)={\rm div}(W(\nabla \psi(u,t)+\bar \rho)), \quad (u,t)\in[0,1]^2\times [0,\infty)\\
\psi(u,0)=\psi_0(u).
     \end{array}
\right.
   \end{eqnarray} 
The function $\psi$ is  $C^{2,1}$  ($C^2$ in space and $C^1$ in time) and its gradient $\nabla\psi$ is continuous in time.
   Moreover,  $\nabla \psi(u,t)+\bar\rho\in \mathcal A$ for every $u\in [0,1]^2,t\ge0$.
 \item [(ii)] For every $t>0$
   \begin{eqnarray}
     \label{eq:limite}
     \lim_{L\to\infty}D_L(t):= \lim_{L\to\infty}\frac1{L^2}\mathbb E\sum_{u\in\mathcal T_L}\left[
\frac{H(u,t)}L-\psi\left(\frac uL,t\right)
\right]^2=0.
   \end{eqnarray}

  \end{enumerate}
\end{Theorem}

While it is not at first sight obvious that  the PDE \eqref{eq:PDE} is a parabolic equation, it was shown in  \cite[Sec. 3]{LThydro}
that it can be equivalently rewritten as
\begin{eqnarray}
  \label{eq:PDE2}
  \partial_t \psi=\mu(\nabla \psi+\bar\rho)\sum_{i,j=1}^2\sigma_{i,j}(\nabla\psi+\bar\rho)\frac{\partial^2}{\partial_{u_i}\partial_{u_j}}\psi
\end{eqnarray} 
with $\mu(\cdot)>0$ and, for every $\rho\in\mathbb T$, $\{\sigma_{i,j}(\rho)\}_{i,j=1,2}$ a
strictly positive-definite matrix, that is the Hessian
of a strictly convex surface tension function. 
More precisely \cite{LThydro},
\begin{eqnarray}
  \label{eq:4mu}
  \mu(\rho)=\frac1{2\pi} \frac{\sin(\pi\rho_1)\sin(\pi \rho_2)}{\sin(\pi(1-\rho_1-\rho_2))}
\end{eqnarray}
while $\sigma_{i,j}(\rho)=\partial^2_{\rho_i \rho_j}\sigma(\rho)$, 
\begin{eqnarray}
  \label{eq:sigma}
  \sigma(\rho)=\left\{
    \begin{array}{cc}
\frac1\pi\left[
\Lambda(\pi \rho_1)+\Lambda(\pi \rho_2)+\Lambda(\pi(1-\rho_1-\rho_2))\right]\le0,& \rho\in \mathbb T\cup \partial \mathbb T\\
      +\infty & \quad \text{otherwise}
    \end{array}
\right.
\end{eqnarray}
with
\[
\Lambda(\theta)=\int_0^{\theta} \ln (2 \sin(t))dt.
\]
Note that $\sigma(\cdot)$ is the well-known surface tension for the
dimer model on the hexagonal lattice as a function of the slope,
cf. for instance \cite[Sec. 7]{Kenyon}, and $-\Lambda(\cdot)$ is the
so-called Lobachevsky function.  Let us recall that, as well known and discussed for instance in
\cite[Sec. 3]{LThydro}, $\sigma(\cdot)$ is $C^\infty$, strictly
negative and strictly convex in $\mathbb T$, and note that
$\mu(\cdot)$ is also $C^\infty$ and strictly positive in $\mathbb T$.
Moreover, when $\rho$ approaches $\partial\mathbb T$ the function
$\sigma(\rho)$ vanishes, the functions $\mu(\rho)$ and
$\sigma_{i,j}(\rho)$ become singular and $\{\sigma_{i,j}\}_{i,j=1,2}$
loses the strict positive definiteness. Since however
$\nabla\psi+\bar\rho$ stays in $\mathcal A\subset\mathbb T$ at all
times, these potential singularities do not affect the solution of \eqref{eq:PDE2}.

\begin{Remark}
  \label{rem:regolarita}
  Recall that the triangle $\mathbb T$ is open.  Condition
  \eqref{eq:etazero} means that we are requiring that
  $\nabla\psi_0+\bar\rho$ is uniformly away from $\partial\mathbb T$,
  the boundary of the set of allowed slopes. This is not just a
  technical assumption; indeed the coefficients $\mu(\cdot)$ and
  $\sigma_{i,j}(\cdot)$ of the hydrodynamic PDE \eqref{eq:PDE2} are
  singular on $\partial \mathbb T$ and it is not clear that its
  solution is well defined in general if \eqref{eq:etazero} fails for
  the initial condition.  On the other hand, the $C^2$ condition can
  certainly be relaxed, along the following lines. Assume $\psi_0$ is
  Lipschitz and satisfies \eqref{eq:etazero}; approximate it, within
  distance $\epsilon$ in sup norm, by $C^2$ functions $\psi^\pm$, with
  $\psi^-\le \psi_0\le\psi^+$, that verify \eqref{eq:etazero} for some
  compact sets $\mathcal A^\pm$, close to $\mathcal A$ in Hausdorff
  distance. Theorem \ref{th:main} holds for initial conditions
  $\eta_0^\pm$ that tend to $\psi^\pm$; since the dynamics is
  attractive, i.e. preserves stochastic ordering between height
  profiles, the hydrodynamic limit for initial condition $\eta_0$ can
  be obtained by letting $\epsilon\to0$ after $L\to\infty$.  To avoid
  overloading this work, we prefer not to work out in full detail the
  argument we just outlined, and to state the main result under the
  $C^2$ assumption.
\end{Remark}

\subsection{Short sketch of the proof of the main Theorem}
The proof of item (ii) of Theorem \ref{th:main} is the object of
Sections \ref{sec:DLt} to \ref{sec:calcolo} and, in its general lines,
follows the ideas of the $H^{-1}$ norm method as in \cite{FunakiSpohn}.  The point is
to study
\begin{equation}
  \label{eq:strategy}
D_L(t)-D_L(0)  
\end{equation}
and to show that it is non-positive in the $L\to\infty$ limit.  The
first step, accomplished in Section \ref{sec:DLt}, is to note that one can
rewrite \eqref{eq:strategy} as  the average, with respect to certain
space-time averaged limiting measures $\nu_{t, (i,j)}$, of a function
depending  only on the local gradients of $H$ and of $\psi$. This step would not
be possible e.g. for the single-flip Glauber dynamics, and the fact
that it works in our case is the signature of the above-mentioned ``gradient condition''.  The
second step, in Section \ref{sec:dec}, is to argue that the measures
$\nu_{t, (i,j)}$ have to be translation invariant Gibbs measures
because of a general entropy production argument. This allows us to
rewrite them as some linear combinations of ergodic, translation
invariant measures with (unknown) density $w_{\nu_{t (i,j)}}$.
Finally, in Section \ref{sec:calcolo}, thanks to the exact solvability
of  the uniform dimer covering model, we rewrite \eqref{eq:strategy}
as an explicit function of the density $w_{\nu_{t, (i,j)}}$. A non-trivial algebraic identity then implies that this function is
non-positive independently of $w_{\nu_{t, (i,j)}}$, which concludes
the proof. This identity is related to the fact that the limit PDE
contracts the $\mathbb L^2$ distance which, as we mentioned in the
introduction, is also a remarkable and a-priori not obvious property of our model.

\section{The limit PDE}
\label{sec:PDE}
Item (i) in Theorem \ref{th:main} is a consequence of rather
  classical results in the theory of non-linear parabolic PDEs, cf. for instance \cite{Liberman}. We still discuss
  it briefly, mostly to show why the fact that the coefficients of the equation are singular on the boundary of  $\mathbb T$ does not affect the regularity of the solution.

\medskip
Consider the Cauchy problem
\begin{eqnarray}
  \label{eq:PDE3}
\left\{
  \begin{array}{l}
  \partial_t \phi=\kappa(\nabla \phi)\sum_{i,j=1}^2g_{i,j}(\nabla\phi)\frac{\partial^2}{\partial_{u_i}\partial_{u_j}}\phi, \;u\in[0,1]^2, t\ge0\\
\phi(u,0)=\phi_0(u), 
  \end{array}
\right.
\end{eqnarray} 
where $\phi_0(\cdot)$ is a $C^2$ periodic function on  $[0,1]^2$. Here, 
 $\kappa(\cdot)$ is a strictly positive and $C^\infty$ function on $\mathbb R^2$, while $g_{i,j}(x)=\partial^2_{x_i x_j}g(x)$, for a strictly convex, $C^\infty$ function $g:\mathbb R^2\to \mathbb R$. We further assume that, letting 
$g^{(2)}(x):=\{g_{i,j}(x)\}_{i,j=1,2}$, one has
\[
c_- \mathbb I\le g^{(2)}(x),
\]
with $\mathbb I$ the $2\times 2$ identity matrix and $c_->0$ 
and, finally, that
\[
|p|^3 |D g_{i,j}(p)|\stackrel{|p|\to\infty}=O(|p|^2)
\]
with $D
g_{i,j}$ the gradient of
$g_{i,j}$.
Then, it is known (see \cite[Th. 12.16]{Liberman} and the remarks following it) that \eqref{eq:PDE3}
admits a unique global classical solution $\phi$
that belongs to a certain H\"older space $H_a,a>2$
and in particular that $\phi$
is at least $C^2$
in space, $C^1$
in time and its gradient $\nabla\phi$
is continuous in time (in reality, using the $C^\infty$
regularity of the coefficients of the equation one may bootstrap the
argument to get $C^\infty$
regularity for $\phi$,
but we do not need this).

Assume that $\phi_0$ satisfies the same condition as
$\psi_0$ in \eqref{eq:etazero}. A standard comparison argument gives
that
$\phi(\cdot,t)$ still satisfies \eqref{eq:etazero} at all later
times. Namely, write the convex set $\mathcal
A-\bar\rho$ (the translation of $\mathcal A$ by $-\bar\rho$) as the intersection of half-planes:
\begin{eqnarray}
  \label{eq:convA}
  \mathcal A-\bar\rho=\cap_{n\in \mathbb S^1}\{x\in \mathbb R^2: x\cdot n\le f(n)\}
\end{eqnarray}
 and for any $n\in\mathbb S^1,a>0$ define
\[
\hat \phi(u,t):=\phi(u+n a,t)-a f(n),
\]
which is still a smooth solution of \eqref{eq:PDE3}, this time with initial condition $\phi_0(\cdot+n a)-a f(n)$.
Since at time zero $\hat \phi\le \phi$, by the maximum principle the 
same holds at later times and therefore $\nabla\phi(u,t)+\bar\rho\in\mathcal A$.

Now let us go back to our PDE \eqref{eq:PDE} and let us recall that it
is equivalent to \eqref{eq:PDE2}.
Define $g:\mathbb R^2\mapsto \mathbb R$ and
$\kappa:\mathbb R^2\mapsto \mathbb R$ such that $g$ and $\kappa$
coincide with $\sigma$ and $\mu$ respectively on $\mathcal A$, while
at the same time $g$ and $\kappa$ satisfy the smoothness and
convexity/positivity assumptions formulated just after
\eqref{eq:PDE3}.  From the discussion above we know that the solution
of the Cauchy problem \eqref{eq:PDE3} with $\phi_0\equiv \psi_0$ is
smooth and $\nabla\phi(u,t)\in \mathcal A$ for every
$u,t$. Since $g\equiv \sigma$ and $\kappa\equiv \mu$ in
$\mathcal A$, we deduce that $\phi$ also solves
\eqref{eq:PDE2} (and therefore \eqref{eq:PDE}).

\section{Computation of $D_L(t)$ in terms of limit measures}
\label{sec:DLt}

We start the bulk of the work, i.e. the proof of claim (ii) of Theorem \ref{th:main}.
The goal of this section is to prove the upper bound  \eqref{eq:inco} for $\limsup_L D_L(t)$ in terms of certain 
space-time averaged measures $\nu_{t,(i,j)}$.
\subsection{Some preliminary definitions}
We need some notation: given a function $f:\mathcal T_L\to\mathbb R$ we let \[\hat\nabla_i f(u)=f(u)-f(u-\hat e_i), i=1,2\]
and $\hat\nabla f(u)=(\hat\nabla_1 f(u),\hat\nabla_2 f(u))$.
Moreover, given $\eta\in\Omega_{\bar\rho^{(L)}}$ and $u\in\mathcal T_L$, we let
\begin{equation}
  \label{eq:eps}
   \epsilon(u,\eta)={\bf 1}_{\{b_1(u),b_2(u)\}\subset \eta}-{\bf 1}_{\{b_1(u-\hat e_1),b_2(u-\hat e_2)\}\subset\eta}\in\{0,-1,+1\},
\end{equation}
with $b_i(u)$ as in Definition \ref{def:Ai}.
It is immediately seen that the events \[\{b_1(u),b_2(u)\}\subset \eta\] and
\[\{b_1(u-\hat e_1),b_2(u-\hat e_2)\}\subset\eta\] are
 mutually exclusive.

For any vertex $u\in\mathcal T_L$ such that $\epsilon(u,\eta)\ne0$, we define 
\begin{eqnarray}
  \label{eq:ka}
  k(u,\eta)=\max\{n\ge 1:\epsilon(u+(n-1)\epsilon(u,\eta)\hat e_3,\eta)=\epsilon(u,\eta)\}\ge 1.
\end{eqnarray}
Note that, if $\epsilon(u,\eta)=+1$ (resp. if $\epsilon(u,\eta)=-1$) then $k(u,\eta)$ is the smallest integer such that there is a dimer of type $3$ at $b_3(u+(n-1)\hat e_3)$ (resp. at
$b_3(u-n\hat e_3)$). See Fig. \ref{fig:k}.

\begin{figure}[h]
  \includegraphics[width=5cm]{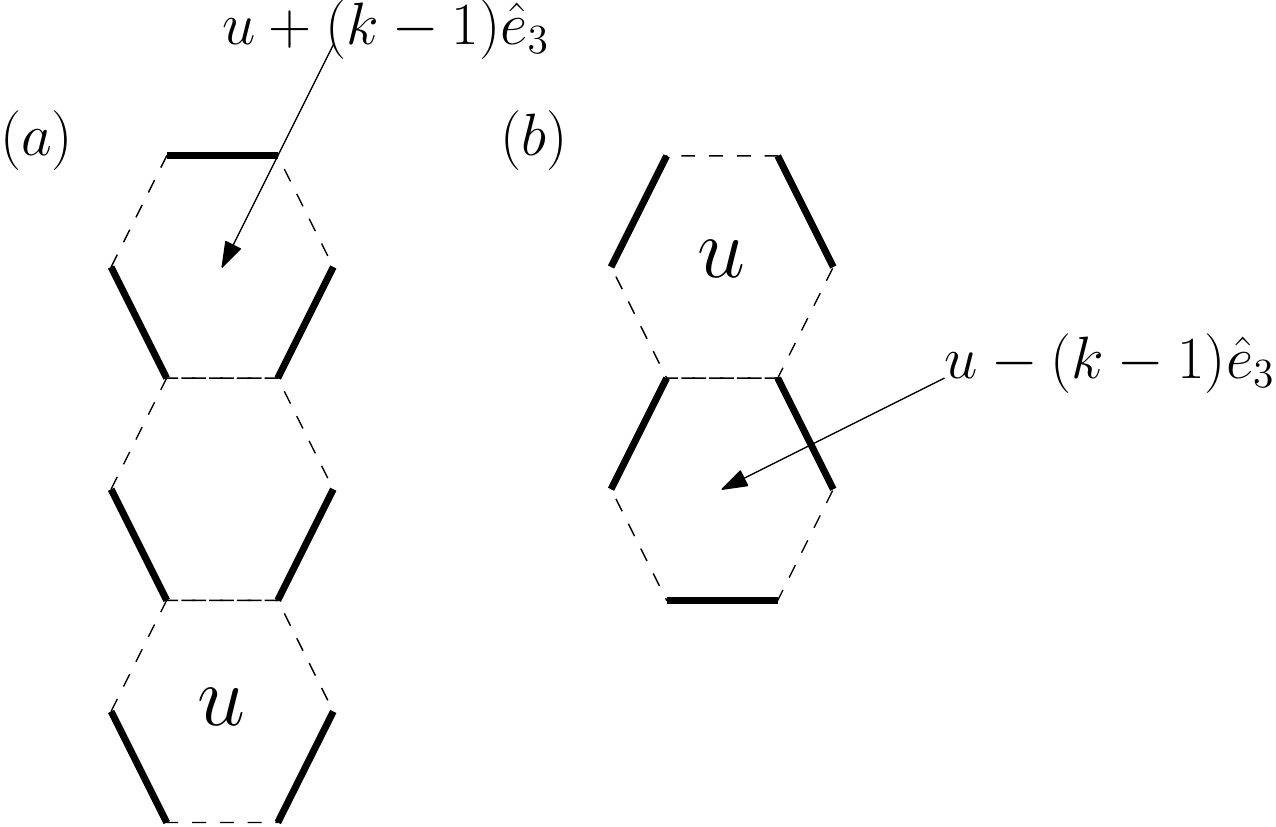}
  \caption{(a): here, $\epsilon(u,\eta)=+1$ and $k=k(u,\eta)=3$ since $k$ is the smallest integer such that 
$\epsilon(u+(k-1)\hat e_3,\eta)=1$ while $\epsilon(u+k\hat e_3,\eta)\ne 1$. (b): in this case, $\epsilon(u,\eta)=-1$ and $k=k(u,\eta)=2$. }
\label{fig:k}
\end{figure}

For later convenience, it is useful to remark the following:
\begin{Lemma} For every configuration $\eta\in \Omega_{\bar\rho^{(L)}}$ and $u\in \mathcal T_L$, one has the identity
\label{lemma:epsilon}
\begin{eqnarray}
  \label{eq:epsilon}
  \epsilon(u,\eta)=\frac12\left[
F(u+\hat e_3,\eta)-F(u,\eta)+\Delta H_\eta(u)
\right]
\end{eqnarray}  
where
\begin{eqnarray}
  \label{eq:eps2}
\Delta H_\eta(u)=\sum_{i=1}^2\left[H_\eta(u+\hat e_i)+H_\eta(u-\hat e_i)\right]-4H_\eta(u)\\=
\hat\nabla_1(\hat\nabla_1 H_\eta)(u+\hat e_1)+\hat\nabla_2(\hat\nabla_2 H_\eta)(u+\hat e_2)
\end{eqnarray}
and
\begin{eqnarray}
  \label{eq:F}
  F(u,\eta)=|{\bf 1}_{b_1(u)\in\eta}-{\bf 1}_{b_2(u)\in\eta}|.
\end{eqnarray}
\end{Lemma}
Note that, while $\mathcal T$ is the triangular lattice, $\Delta$ in
\eqref{eq:eps2} is the ordinary $\mathbb Z^2$ Laplacian with respect
to the coordinate axes $\hat e_1,\hat e_2$.
\begin{proof}[Proof of Lemma \ref{lemma:epsilon}]
Note first of all the trivial identities
\begin{eqnarray}
\label{eq:Ff}
  |{\bf 1}_{b\in\eta}-{\bf 1}_{b'\in\eta}|={\bf 1}_{b\in\eta}+{\bf 1}_{b'\in\eta}-2
{\bf 1}_{\{b,b'\}\subset\eta}
\end{eqnarray}
and
\begin{eqnarray}
  \label{eq:Ff2}
  \Delta H_\eta(u)=({\bf 1}_{b_1(u)\in\eta}-{\bf 1}_{b_1(u-\hat e_1)\in \eta})+
({\bf 1}_{b_2(u)\in\eta}-{\bf 1}_{b_2(u-\hat e_2)\in\eta}).
\end{eqnarray}
Then,  the r.h.s. of \eqref{eq:epsilon} equals
\begin{multline}
\frac{1}{2}|{\bf 1}_{b_1(u + \hat e_3)\in\eta}-{\bf 1}_{b_2(u+ \hat e_3)\in\eta}| - \frac{1}{2}({\bf 1}_{b_1(u)\in\eta}+{\bf 1}_{b_2(u)\in\eta}-2
{\bf 1}_{\{b_1(u),b_2(u)\}\subset\eta}) \\
+ \frac{1}{2}({\bf 1}_{b_1(u)\in\eta}-{\bf 1}_{b_1(u-\hat e_1)\in \eta})+
\frac{1}{2}({\bf 1}_{b_2(u)\in\eta}-{\bf 1}_{b_2(u-\hat e_2)\in\eta})
\end{multline}
and using \eqref{eq:Ff} again for $b=b_1(u-\hat e_1),b'=b_2(u-\hat e_2)$, we obtain
\begin{multline}
  \label{eq:Ff3}
 {\bf 1}_{\{b_1(u),b_2(u)\}\subset \eta}-{\bf 1}_{\{b_1(u-\hat e_1),b_2(u-\hat e_2)\}\subset \eta}\\+
\frac12\left[
|{\bf 1}_{b_1(u+\hat e_3)\in\eta}-{\bf 1}_{b_2(u+\hat e_3)\in\eta}|-|{\bf 1}_{b_1(u-\hat e_1)\in\eta}-{\bf 1}_{b_2(u-\hat e_2)\in\eta}|
\right].
\end{multline}
The first line of \eqref{eq:Ff3} is exactly $\epsilon(u,\eta)$ and the
second line is identically zero (this is easily checked in
all the finitely many allowed dimer configurations of edges
$b_1(u+\hat e_3), b_2(u+\hat e_3),b_1(u-\hat e_1),b_2(u-\hat
e_2),b_3(u)$).
\end{proof}

\subsection{Computation of $d/dt D_L(t)$}
 By our assumption on the initial condition we know that \eqref{eq:limite}  holds for $t=0$. Then, we differentiate $D_L(t)$ with respect to $t$ and integrate back with the hope that, for every $t>0$,
\begin{eqnarray}
  \label{eq:spero}
  \limsup_{L\to\infty} \int_0^t ds\,\frac d{ds}D_L(s)\le 0.
\end{eqnarray}

We first notice that the time derivative of $D_L(t)$ can be written as the sum of four terms:
\begin{gather}
  \label{eq:3}
  \frac d{dt}D_L(t)=\sum_{i=1}^4B_i(t)\\
\label{eq:B1bis}
B_1(t)=\frac d{dt}\frac1{L^2}\mathbb E\sum_{u\in\mathcal T_L}\frac{ H^2(u,t)}{L^2}
 \\ 
B_2(t)= -\frac2{L^2}\sum_{u\in\mathcal T_L}\psi(u/L,t)\frac d{dt}\mathbb E\frac{ H(u,t )}L
\\
B_3(t)=-\frac2{L^2}\sum_{u\in\mathcal T_L}\mathbb E\frac{ H(u,t )}L\partial_t\psi(u/L,t)
\\
B_4(t)=\frac2{L^2}\sum_{u\in\mathcal T_L}\psi(u/L,t)\partial_t\psi(u/L,t).
\end{gather}
We  put 
$ \Psi(u,t):=L\psi(u/L,t)$, with $u\in\mathcal T_L$ and $\psi$ the solution of \eqref{eq:PDE}.
Then we re-express each $B_i(t)$ using the definition of the transition rates:
\begin{Proposition}
\label{prop:primaderivata}
  \begin{gather}
    B_1(t)=\frac1{L^2}\mathbb E\sum_{u\in\mathcal T_L}\left(
\frac{|\epsilon(u,t)|}2-\bar\rho_3^{(L)}|\epsilon(u,t)|\frac{k(u,t)-1}2+ \epsilon(u,t) H(u,t)
\right) \label{eq:B1}
\\
B_2(t)=-\frac1{L^2}\sum_{u\in\mathcal T_L}\mathbb E \left\{\Psi\left(u,t\right)\epsilon(u,t)+\frac{\epsilon(u,t)}{k(u,t)}\sum_{j=0}^{k(u,t)-1}\left[\Psi\left(u+j\epsilon(u,t)\hat e_3\right)-\Psi\left(u,t\right)\right]\right\}  \label{eq:B2}
\\
B_3(t)=\frac2{L^2}\sum_{u\in\mathcal T_L}W(\hat \nabla\Psi(u,t)+\bar \rho)\cdot\mathbb E\hat\nabla  H(u,t )+o(1)  \label{eq:B3}
\\
B_4(t)=-\frac2{L^2}\sum_{u\in\mathcal T_L}\hat \nabla \Psi(u,t)\cdot W(\hat \nabla\Psi(u,t)+\bar \rho)+o(1) \label{eq:B4}
\end{gather}
  where the error terms $o(1)$ tend to zero as $L\to\infty$, uniformly
  on compact time intervals, and $k(u,t),\epsilon(u,t)$ are short-hand
  notations for $k(u,\eta(t)),$ $\epsilon(u,\eta(t ))$.
\end{Proposition}

\begin{proof}[Proof of Proposition \ref{prop:primaderivata}]
  We remark that, if a particle $p$ moves downward to an edge $b$ and
  $u\in\mathcal T_L$ is the unique vertex such that $b=b_3(u-\hat e_3)$, then
  necessarily $\epsilon(u,\eta)=+1$.  Similarly, if a particle $p$
  moves upward to an edge $b$ and $u\in\mathcal T_L$ is the unique
  vertex such that $b=b_3(u)$, then necessarily
  $\epsilon(u,\eta)=-1$.
  In both cases the length of such a jump is exactly $k(u,\eta)$, so
  that the jump occurs with rate $L^2/(2k(u,\eta))$, and (see Remark \ref{rem:1}) the height function
  $H$ changes by $\epsilon(u,\eta)$ at faces labelled $u+j\epsilon(u,\eta)\hat e_3$,
  $j=0,\dots,k(\epsilon,\eta)-1$.

One deduces then
  \begin{multline}
    \label{eq:b1}
    B_1(t)=\frac1{L^2}\sum_{u\in\mathcal T_L}\mathbb E \frac{|\epsilon(u,t)|}{2k(u,t)}\sum_{j=0}^{k(u,t)-1}
\Bigl[\left(H(u+j \epsilon(u,t)\hat e_3,t)+\epsilon(u,t)\right)^2\\-H(u+j \epsilon(u,t)\hat e_3,t)^2\Bigr]\\
=\frac1{L^2}\sum_{u\in\mathcal T_L}\mathbb E \frac{|\epsilon(u,t)|}{2k(u,t)}\sum_{j=0}^{k(u,t)-1}\left(
2\epsilon(u,t)H(u+j \epsilon(u,t)\hat e_3,t)+1
\right),
  \end{multline}
  where the factor $|\epsilon(u,t)|$ is there to select those $u$ for
  which $\epsilon(u,t)\ne0$ and a transition is possible.  On the
  other hand,
\begin{eqnarray}
\nonumber
  H(u+j \epsilon(u,t)\hat e_3,t)=H(u,t)-j \bar\rho^{(L)}_3\epsilon(u,t) \;\text{for every } j=0,\ldots, k(u,t)-1
\end{eqnarray}
as follows from \eqref{eq:gradh2} for $i=3$, since no horizontal dimer is crossed by the vertical path from $u$ to $u+\epsilon(u,t)(k(u,t)-1)\hat e_3$.
Therefore, summing over $j$, one finds \eqref{eq:B1} as desired.

Similarly,
\begin{eqnarray}
  \label{eq:b2}
  B_2(t)=-\frac2{L^2}\sum_{u\in\mathcal T_L}\mathbb E
 \frac{|\epsilon(u,t)|}{2k(u,t)}\sum_{j=0}^{k(u,t)-1}\epsilon(u,t)\Psi(u+j \epsilon(u,t)\hat e_3,t)
\end{eqnarray}
which is the same as \eqref{eq:B2}.

As for \eqref{eq:B3}, just note that
\begin{eqnarray}
  \label{eq:b3}
  B_3(t)=-\frac2{L^2}\sum_{u\in\mathcal T_L}\mathbb E \frac{H(u,t)}L\left.{\rm div}(W(\nabla\psi+\bar\rho))\right|_{(u/L,t)}.
\end{eqnarray}
Since $\psi$ is $C^2$
in space and 
using also $\nabla\psi(\cdot,t)\in\mathcal A\subset \mathbb T$ and smoothness of $W$ in $\mathbb T$, we have
\begin{eqnarray}
\label{eq:b31}
 \left. {\rm div} (W(\nabla \psi+\bar\rho))\right|_{(u/L,t)}=o(1)+L\sum_{j=1}^2 \hat\nabla_j W_j(\hat\nabla\Psi+\bar\rho){(u,t)}
\end{eqnarray}
where the error term is uniform over all finite time intervals. Note
that the second term in the r.h.s. of \eqref{eq:b31} is $O(1)$ and not
$O(L)$, because the discrete gradient $\hat\nabla_j$ is essentially $1/L$ times the
continuous gradient.  A discrete summation by parts then gives
\begin{eqnarray}
  \label{eq:b32}
    B_3(t)=\frac2{L^2}\sum_{u\in\mathcal T_L}\mathbb E \hat\nabla H(u,t)\cdot W(\hat\nabla \Psi(u,t)+\bar\rho)\\
+o(1)\times \frac1{L^2}\sum_{u\in\mathcal T_L} \mathbb E\frac{H(u,t)}L.
\end{eqnarray}
In order to get \eqref{eq:B3} it suffices to prove that the sum
multiplying $o(1)$ is of order $1$ with respect to $L$, uniformly on
finite time intervals.  In fact, one has even more: the sum in
question does not depend on time at all (and at time zero it is $O(1)$
because by definition  the height function is bounded by $L$).  Indeed, the same steps that
led to \eqref{eq:B2} give
\begin{eqnarray}
\label{eq:0}
\frac{d}{dt}  \frac1{L^2}\sum_{u\in\mathcal T_L} \mathbb E\frac{H(u,t)}L=\frac1{2L}\sum_{u\in\mathcal T_L}\mathbb E \epsilon(u,t)
\end{eqnarray}
(one just needs to replace $-2\psi(\cdot,\cdot)$ by the constant $1$ in
\eqref{eq:B2}).  Using Lemma \ref{lemma:epsilon} and summation by
parts, the r.h.s. of \eqref{eq:0} is immediately seen to be zero.

Finally, the proof of \eqref{eq:B4} is simpler than that of \eqref{eq:B3} (it uses only smoothness of the solution of the PDE and summation by parts, while the Markov process does not appear) and is left to the reader.
\end{proof}
Note that all terms in the sums defining $B_1(t),\dots,B_4(t)$ are of
order $1$ (we will see later that $k(u,t)$ is typically $O(1)$, see
Proposition \ref{prop:contrk}), except for $\epsilon(u,t)H(u,t)$ and
$\epsilon(u,t)\Psi(u,t)$, that a priori are of order $L$. However, the
spatial sum of these terms can be rewritten in a more convenient form,
from which it is clear that  it is of the good order of magnitude:
\begin{Proposition}
\label{prop:id}
One has the following identities:
  \begin{multline}\label{eq:id1}
    \sum_{u\in\mathcal T_L}  \epsilon(u,t) H(u,t) \\= 
    \frac{\bar\rho^{(L)}_3}2 \sum_{u\in\mathcal T_L}F(u,\eta(t))
-\frac1{2}\sum_{u\in\mathcal T_L} \left(
(\hat\nabla_1  H(u,t))^2+(\hat\nabla_2  H(u,t))^2
\right)
  \end{multline}
and
\begin{multline}
\label{eq:id2}
    \sum_{u\in\mathcal T_L}  \epsilon(u,t)\Psi(u,t) =
    \frac{1}2 \sum_{u\in\mathcal T_L}F(u,\eta(t))(\hat\nabla_1\Psi(u,t)+\hat\nabla_2\Psi(u,t))\\
-\frac1{2}\sum_{u\in\mathcal T_L} \left(
\hat\nabla_1 H(u,t)\hat \nabla_1\Psi(u,t)+\hat\nabla_2 H(u,t)\hat \nabla_2\Psi(u,t)
\right)+o(L^2)
\end{multline}
where $F(u,\eta)$ was defined in \eqref{eq:F} and $o(L^2)$ is uniform on bounded time intervals.
\end{Proposition}
\begin{proof}[Proof of Proposition \ref{prop:id}]
 Let us show the first identity \eqref{eq:id1}.
We have, from Lemma \ref{lemma:epsilon} and a summation by parts,
\begin{multline}
  \sum_{u\in\mathcal T_L}\epsilon(u,t)H(u,t)\\=\frac12\sum_{u\in\mathcal T_L}\left[F(u+\hat e_3,\eta(t))-F(u,\eta(t))+\sum_{j=1,2}\hat\nabla_j\hat\nabla_j H(u+\hat e_j,t)\right]H(u,t)\\=-\frac12\sum_{u\in\mathcal T_L}\left[
F(u,\eta(t))(H(u,t)-H(u-\hat e_3,t))+\hat\nabla H(u,t)\cdot \hat\nabla H(u,t)
\right].
\end{multline}
Next, recall \eqref{eq:Ff} and (cf. \eqref{eq:gradh2})
\begin{eqnarray}
\label{eq:ac}
  H(u,t)-H(u-\hat e_3,t)={\bf 1}_{b_3(u-\hat e_3)\in\eta(t)}-\bar\rho^{(L)}_3.
\end{eqnarray}
Since the event $\{b_3(u-\hat e_3)\in \eta\}$ is incompatible with both
$\{b_1(u)\in\eta\}$ and $\{b_2(u)\in\eta\}$ (the two events appearing in the definition of $F(u,\eta)$), the indicator function
${\bf 1}_{b_3(u-\hat e_3)\in\eta(t)}$ in \eqref{eq:ac} can be omitted
and \eqref{eq:id1} follows.

The proof of \eqref{eq:id2} is analogous. The error term $o(L^2)$ comes from approximating 
\[
\Psi(u-\hat e_3,t)-\Psi(u,t)=\hat\nabla_1\Psi(u,t)+\hat\nabla_2\Psi(u,t)+o(1).
\]
\end{proof}

Altogether, given that 
\begin{eqnarray}
  \label{eq:gradi}
  \hat\nabla_iH(u,t)={\bf 1}_{b_i(u-\hat e_i)\in\eta(t)}-\bar\rho^{(L)}_i, \;i=1,2
\end{eqnarray}
 we have obtained:
\begin{Proposition}
\label{prop:Ci}
  \begin{equation}
    \label{eq:Bi}
    B_i(t)=\frac1{L^2}\mathbb E\sum_{u\in\mathcal T_L} C_i(\eta_{-u}(t),\Psi_{-u}(t))+o(1)
  \end{equation}
where $\eta_{-u}$ and $\Psi_{-u}$ denote the configurations $\eta$ and $\Psi$ translated by $-u$, $o(1)$ is uniform on bounded time intervals  and 
we set
\begin{align}
\label{eq:C1}
 & C_1(\eta,\Psi(t)):=C_1(\eta)  =\frac{|\epsilon(0,\eta)|}2-\bar\rho^{(L)}_3|\epsilon(0,\eta)|\frac{k(0,\eta)-1}2
+\frac{\bar\rho_3}2 F(0,\eta)
\\\nonumber&-\frac1{2}\left(
{\bf 1}_{b_1(0)\in\eta}(1-2\bar\rho_1)+\bar\rho_1^2+{\bf 1}_{b_2(0)\in\eta}(1-2\bar\rho_2)+\bar\rho_2^2
\right)
\\
 & C_2(\eta, \Psi(t)) =-\frac{\epsilon(0,\eta)}{k(0,\eta)}\sum_{j=0}^{k(0,\eta)-1}\left[ \Psi\left(j\epsilon(0,\eta)\hat e_3,t\right)- \Psi\left(0,t\right)\right] 
  \\
& \quad  -\frac{1}2F(0,\eta)(\hat\nabla_1\Psi(0,t)+\hat\nabla_2\Psi(0,t)) \nonumber
  \\
&  \quad +\frac1{2} \left(({\bf 1}_{b_1(0)\in\eta}-\bar\rho_1)\hat \nabla_1\Psi(0,t)+({\bf 1}_{b_2(0)\in\eta}-\bar\rho_2)\hat \nabla_2\Psi(0,t)
\right) \nonumber
\\
& C_3(\eta, \Psi(t)) =2W(\hat \nabla \Psi(0,t)+\bar \rho)\cdot({\bf 1}_{b_1(0)\in\eta}-\bar\rho_1,{\bf 1}_{b_2(0)\in\eta}-\bar\rho_2)
  \\
&  C_4(\eta, \Psi(t)) =-2\hat \nabla  \Psi(0,t)\cdot W(\hat \nabla \Psi(0,t)+\bar \rho).
\end{align}
\end{Proposition}
The error terms $o(1)$ come from those in Propositions
\ref{prop:primaderivata} and \ref{prop:id}, and also from harmless approximations of the
type
\begin{eqnarray}
\bar\rho^{(L)}_i=\bar\rho_i+o(1),\quad \hat\nabla_j \Psi(u,t)=\hat\nabla_j \Psi(u-\hat e_j,t)+o(1).
\end{eqnarray}

Note that $C_i(\eta, \Psi(t))$ depends only on the \emph{gradients} of
$ \Psi(t)$ (similarly, $C_i$ depends on $\eta$, i.e. on the
\emph{gradients} of the height function associated to
$\eta$). Actually, $C_1$ is independent of $\Psi$ and $C_4$ is
independent of $\eta$, while $C_2,C_3$ depend non-trivially on both variables.  
Also observe that only the $C_1$ term depends explicitly on $L$, through the term $\bar{\rho}^{(L)}_3 |\epsilon(0, \eta)|\frac{k(0, \eta) -1}{2}$.

Recall that we want to prove \eqref{eq:spero}. Integrating $d/ds D_L(s)$ from time $0$ to time $t$, we find from Proposition \ref{prop:Ci}
\begin{multline}
\label{eq:where}
  D_L(t)=D_L(0)+
 \frac t{L^2}\sum_{u\in \mathcal T_L}\frac1t\int_0^t \mathbb E\sum_{i=1}^4
C_i(\eta_{-u}(s), \Psi_{-u}(s)) \, \mathrm{d}s+o(1).
\end{multline}
This expression involves  a triple average on space,
time and on the realization of the process.  
Recall also that, while we do not indicate this explicitly, both the law
$\mathbb E$ and the function  $\Psi(\cdot,s)$ depend on
$L$. 

\medskip

Let $\Omega$ be the set of all dimer
coverings of $\mathcal H$.  We put the  product topology on
$\Omega$, which makes it compact: as a metric we can take for instance
\[d(\eta,\eta')=\sum_{e\in \mathcal H}2^{-|e|}|{\bf 1}_{e\in
  \eta}-{\bf 1}_{e\in \eta'}|
\]
where the sum over $e$ runs over all edges of $\mathcal H$ and $|e|$
is the distance between $e$ and the origin. 
Given this, we observe that
\[
 \sum_{i=1}^4 C_i(\eta,\Psi(s))=C(\eta,\Psi(s))+U^{(L)}(\eta,\Psi(s))\]
where $C(\eta,\Psi(s))$
 is a 
  bounded and continuous function on $\Omega$
 while $U^{(L)}(\eta,\Psi(s))$ is  the unbounded term
\begin{gather}
\label{eq:Fatou}
U^{(L)}(\eta,\Psi(s))\\\nonumber=  -\bar\rho^{(L)}_3|\epsilon(0,\eta)|\frac{k(0,\eta)-1}2-\frac{\epsilon(0,\eta)}{k(0,\eta)}\sum_{j=0}^{k(0,\eta)-1}\left[ \Psi\left(j\epsilon(0,\eta)\hat e_3,s\right)- \Psi\left(0,s\right)\right].
\end{gather}
Note in particular that, as indicated by the notation, there is no dependence on $L$ in the function $C$. 
The function $U^{(L)}(\eta,\Psi(s))$ instead depends explicitly on $L$ through $\bar\rho^{(L)}_3$.
The goal of the rest of this section is to replace $U^{(L)}$ with a bounded, continuous and $L$-independent  function $U_{M}$, where
$M$ is a cut-off parameter that will be sent to infinity at the very end.

First of all we claim:
\begin{Lemma}
  \label{Uneg}
One has $U^{(L)}(\eta,\Psi(s))\le0$ as soon as $L$ is larger than some $L_0$ independent
of $\eta,s$.
\end{Lemma}
\begin{proof}[Proof of Lemma \ref{Uneg}]
  We can assume that $\epsilon(0,\eta)=\pm1$: indeed, if
$\epsilon(0,\eta)=0$ then $U^{(L)}(\eta,\Psi(s))=0$ and there is nothing to
prove. 
 Since $\nabla\psi(u,t)+\bar \rho\in\mathcal A$ for every $u\in[0,1]^2,t\ge0$, we know that $\Psi(t)$ satisfies
\begin{gather}
\label{eq:Fdelta}
\left\{
  \begin{array}{ll}
\delta\le  \hat\nabla_i \Psi(u,t)+\bar\rho_i\le 1-\delta, & i=1,2,\\\delta\le  \hat\nabla_1 \Psi(u,t)+\hat\nabla_2 \Psi(u,t)+\bar\rho_1+\bar\rho_2\le 1-\delta   &,
  \end{array}
\right.
\end{gather}
for some $\delta>0$ independent of $t,u$.  
Therefore,
\begin{gather}
  -\epsilon(0,\eta)\left[ \Psi\left(j\epsilon(0,\eta)\hat e_3,s\right)- \Psi\left(0,s\right)\right]\le j(1-\delta-\bar\rho_1-\bar\rho_2)\\
=|\epsilon(0,\eta)|j(\bar\rho_3-\delta).
\end{gather}
Summing over $j$, the r.h.s. of \eqref{eq:Fatou} is upper bounded by 
\begin{eqnarray}
-|\epsilon(0,\eta)|\frac{k(0,\eta)-1}2(\bar\rho^{(L)}_3-\bar\rho_3+\delta)
\end{eqnarray}
which is negative for $L$ large enough given that $\bar\rho^{(L)}_3\to\bar\rho_3$.  
\end{proof}
As a consequence, for every $M>0$ and assuming $L>L_0$ we have
\begin{eqnarray}
  \label{eq:continuo}
U^{(L)}(\eta,\Psi(s))\le U^{(L)}(\eta,\Psi(s)){\bf 1}_{|\epsilon(0,\eta)|k(0,\eta)\le M} =U_{M}(\eta,\Psi(s))+\delta_L(M)  
\end{eqnarray}
where 
\begin{eqnarray}
  \label{eq:uinf}
  U_{M}(\eta,\Psi(s))=U(\eta,\Psi(s)){\bf 1}_{|\epsilon(0,\eta)|k(0,\eta)\le M},
\end{eqnarray}
with $U(\eta,\Psi(s))$ defined as $U^{(L)}(\eta,\Psi(s))$ except that $\bar\rho^{(L)}$ is replaced by $\bar \rho$,
while
\[
  \delta_L(M)=(\bar\rho_3-\bar\rho^{(L)}_2)|\epsilon(0,\eta)|\frac{k(0,\eta)-1}2{\bf 1}_{|\epsilon(0,\eta)|k(0,\eta)\le M}
\]
so that $|\delta_L(M)|\le |\bar\rho_3-\bar\rho^{(L)}_2|M$ which tends
to zero as $L\to\infty $ (not uniformly in $M$). Note that $U_{M}$ is
continuous in $\eta$. Indeed, the mapping
\[\eta\mapsto M\wedge |\epsilon(0,\eta)|k(0,\eta)\] is continuous (its
value is determined by the configuration of $\eta$ in a window of size
$M$ around the origin).

Altogether, we have obtained:
\begin{Proposition} As soon as $L\ge L_0$, we have for every $M>0$
  \begin{multline}
\label{eq:lbella}
    D_L(t)\le D_L(0)+\delta_L(M) \\
+ \frac t{L^2}\sum_{u\in \mathcal T_L}\frac1t\int_0^t \mathbb E \left[C(\eta_{-u}(s),
 \Psi_{-u}(s))+U_M(\eta_{-u}(s),
 \Psi_{-u}(s))\right] \, \mathrm{d}s.
  \end{multline}
The functions $C(\cdot,\Psi(t))$ and $U_M(\cdot,\Psi(t))$ are continuous and bounded on $\Omega$, and the error term $\delta_L(M)$ tends to zero 
as $L\to\infty$.
\end{Proposition}
\subsection{Space-time discretization}
\label{sec:discretization}
Divide the torus $\mathcal T_L$ into $N^2$ disjoint square boxes $B_j,j=(j_1,j_2),  1\le j_1,j_2\le N$:
\[B_j=\{u=(u_1,u_2)\in \mathcal T_L: u_a\in \{L(j_a-1)/N ,\ldots,L j_a/N -1)\}, a=1,2\},\quad\] 
of side $ L/N$ and the time interval $[0,t)$ into sub-intervals \[I_j=[t(j-1)/N,tj/N),\quad j\le N.\] To avoid a plethora of $\lfloor \cdot\rfloor$, we pretend that $N$ and $ L/N$ are both integers.
Let 
\[\mathcal I_N=\{(i,j):1\le i\le N, j=(j_1,j_2) ,1\le j_1,j_2\le N\}.\]
Given $(i,j)\in\mathcal I_N$ let
\begin{eqnarray}
\label{eq:zij}
  z^{(i,j)}=(z_1^{(i,j)},z_2^{(i,j)}):=\nabla \psi(u,s)|_{u=j/N,s=ti/N}:
\end{eqnarray}
thanks to the smoothness properties of  the solution of the PDE stated in point (i) of Theorem \ref{th:main}  we have that the discrete
gradient of $\Psi(\cdot,s)$ inside box $B_j$ and for $s\in I_i$ is
given by $z^{(i,j)}$, up to an  error $\epsilon_N$ that is  $o(1)$ (as $N\to\infty$), uniformly in $(i,j),s,u$.
Therefore, for $u\in B_j$ and $s\in I_i$ we can approximate
\begin{gather}
  \label{eq:appr2}
  U_{M}(\eta_{-u}(s),\Psi_{-u}(s))  =U^{(i,j)}_M(\eta_{-u}(s))+\epsilon_{M,N}\\\label{450}
U^{(i,j)}_M(\eta):=-
  (\bar\rho_3-z_1^{(i,j)}-z_2^{(i,j)})|\epsilon(0,\eta)|\frac{k(0,\eta)-1}2{\bf 1}_{|\epsilon(0,\eta)|k(0,\eta)\le M}
\end{gather}
where, for every fixed $M$,
\begin{eqnarray}
  \label{eq:eNM}
  \limsup_{N\to\infty}\limsup_{L\to\infty}\sup_{(i,j)\in\mathcal I_N, u\in B_j,s\in I_i}|\epsilon_{N,M}|=0
\end{eqnarray}
Similarly, we approximate
\begin{eqnarray}
  \label{eq:appr1}
  C(\eta_{-u}(s),\Psi_{-u}(s))=C^{(i,j)}(\eta_{-u}(s))+\epsilon_N
\end{eqnarray}
and $C^{(i,j)}(\eta_{-u}(s))$ is obtained from
$C(\eta_{-u}(s),\Psi_{-u}(s))$ by replacing every occurrence of
$\hat\nabla\Psi(u,s)$ by $z^{(i,j)}$ (see \eqref{eq:pallosa} for the explicit expression). Note that the functions
$\Omega\mapsto C^{(i,j)}(\eta)$ and $\Omega\mapsto U_M^{(i,j)}(\eta)$
are independent of time.

We can rewrite \eqref{eq:lbella} as
\begin{multline}
\label{eq:where2}
  D_L(t)\le D_L(0)+\frac t{N^3}\sum_{(i,j)\in\mathcal I_N}\, p_{L,t,(i,j)}\left[ C^{(i,j)}(\eta)+U^{(i,j)}_{M}(\eta)\right]\\+\delta_L(M)+\epsilon_{M,N}
\end{multline}
where, for every $f:\Omega\mapsto \mathbb R$ we let
\begin{equation}
\label{eq:pL}
p_{L,t,(i,j)}(f(\eta)):= 
\frac1{( L/N)^2}\sum_{u\in B_j}\frac1{t/N}\int_{I_i} \mathbb E
(f(\eta_{-u}(s))) \, \mathrm{d}s.
\end{equation}

Note that the measure $p_{L,t,(i,j)}$ involves a triple average: over
the time interval $I_j$, over the space window $B_j$ and over the
realization of the process.  Since $\Omega$ is compact, the sequence
of probability measures $\{p_{L,t,(i,j)}\}_{L\ge 1}$ is automatically tight, so
it has sub-sequential limits. Let $\{L_m\}_{m\ge1}$ be a sub-sequence such that
$\{p_{L_m,t,(i,j)}\}_{m\ge1}$  converges weakly to a limit point $\nu_{t,(i,j)}$  for every $(i,j)\in\mathcal I_N$.

We have noted above that both functions $C^{(i,j)}$ and (thanks to the cut-off $M$) $U_M^{(i,j)}$ are bounded and continuous on $\Omega$. 
Therefore, by definition of weak convergence,  we have
\begin{eqnarray}
\label{eq:ofc}
  \lim_{m\to\infty} p_{L_m,t,i.j}\left[C^{(i,j)}(\eta)+U^{(i,j)}_M(\eta)\right]=\nu_{t,(i,j)}\left[C^{(i,j)}(\eta)+U^{(i,j)}_M(\eta)\right].
\end{eqnarray}
In conclusion we have proven:
\begin{Proposition}
For every $N>0$ there exists a sub-sequence $\{L_m\}_{m\ge1}$ such that 
\begin{multline}
  \label{eq:inco}
  \limsup_{m\to\infty}D_{L_m}(t)\\\le 
  \frac t{N^3}\sum_{(i,j)\in\mathcal I_N}\nu_{t,(i,j)}(C^{(i,j)}(\eta)+U^{(i,j)}_M(\eta))+\epsilon_{M,N}
\end{multline}
where $\epsilon_{N,M}$ verifies 
  \begin{eqnarray}
  \label{eq:eNM2}
  \limsup_{N\to\infty}\sup_{(i,j)\in\mathcal I_N, u\in B_j,s\in I_i}|\epsilon_{N,M}|=0.
\end{eqnarray}
\end{Proposition}

\medskip

The following result will be important in the next section:
\begin{Proposition}
\label{prop:contrk}
For every $t>0$ there exists $K(t)<\infty$ such that, for every $N\ge1$ 
  \begin{eqnarray}
    \label{eq:contrk}
    \limsup_{L\to\infty}\frac1{N^3}\sum_{(i,j)\in\mathcal I_N}p_{L,t,(i,j)}(|\epsilon(0,\eta)|k(0,\eta))\le  K(t).
  \end{eqnarray}
As a consequence, for each family of limit points $\nu_{t,(i,j)},(i,j)\in\mathcal I_N$ of $\{p_{L,t,(i,j)}\}_{L\ge1}$ we have 
\begin{eqnarray}
  \label{eq:nuK}
\frac1{N^3}\sum_{(i,j)\in\mathcal I_N}  \nu_{t,(i,j)}(|\epsilon(0,\eta)|k(0,\eta))\le K(t).
\end{eqnarray}

\end{Proposition}

\begin{proof}[Proof of Proposition \ref{prop:contrk}]
Putting together \eqref{eq:B1bis}  and Proposition  \ref{prop:Ci} we see that
\begin{equation}
  \frac1{L^2}\sum_{u\in\mathcal T_L}\mathbb E\frac{H^2(u,t)}{L^2}=  \frac1{L^2}\sum_{u\in\mathcal T_L}\frac{H^2(u,0)}{L^2}+\frac t{N^3}\sum_{(i,j)\in\mathcal I_N} p_{L,t,(i,j)}\left[C_1(\eta)
\right]+o(1).
\end{equation}
Recalling the definition \eqref{eq:C1} of $C_1(\eta)$ as the sum of 
\[
-\bar\rho^{(L)}_3\frac{k(0,\eta)-1}2|\epsilon(0,\eta)|
\] 
plus  a uniformly bounded function $g(\eta)$, 
we deduce that
\begin{gather}
\label{eq:tpl}
\frac1{N^3}\sum_{(i,j)\in\mathcal I_N}  p_{L,t,(i,j)}(|\epsilon(0,\eta)|k(0,\eta))\le \frac {2}{\bar\rho^{(L)}_3}\|g\|_\infty+\frac{2}{L^2\bar\rho^{(L)}_3t}\sum_{u\in\mathcal T_L}\frac{H(u,0)^2}{L^2}+1.
\end{gather}
Since the height function $H(0,u)=H_{\eta_0}(u)$ is uniformly bounded by $L$ and $\bar\rho^{(L)}_3\to\bar\rho_3>0$, we see that the r.h.s. of \eqref{eq:tpl}  is upper bounded independently of $L,N$ and \eqref{eq:contrk} follows.

For every $M>0$ and sub-sequence $\{L_m\}_{m\ge1}$ along which all sequences $p_{L_m,t,(i,j)},(i,j)\in\mathcal I_N$ have a limit $\nu_{t,(i,j)}$, 
\begin{multline}
  \label{eq:f}
 K(t)\ge \limsup_{m\to\infty} \frac1{N^3}\sum_{(i,j)\in\mathcal I_N}p_{L_m,t,(i,j)}(|\epsilon(0,\eta)|k(0,\eta))\\\ge \limsup_{m\to\infty} \frac1{N^3}\sum_{(i,j)\in\mathcal I_N}p_{L_m,t,(i,j)}[M\wedge |\epsilon(0,\eta)|k(0,\eta)]\\=\frac1{N^3}\sum_{(i,j)\in\mathcal I_N}\nu_{t,(i,j)}[M\wedge |\epsilon(0,\eta)|k(0,\eta)]
\end{multline}
where we used the fact that the mapping
$\eta\mapsto M\wedge |\epsilon(0,\eta)|k(0,\eta)$ is continuous as we
mentioned above. By monotone convergence, letting $M\to\infty$, we
deduce \eqref{eq:nuK}.
\end{proof}

\section{Local equilibrium}\label{sec:dec}

The goal of this section is to show how to compute the r.h.s. of \eqref{eq:inco}. The crucial point is that each measure $\nu_{t,(i,j)}$ is a suitable linear combination of translation invariant, ergodic Gibbs states. This is the 
content of Theorem \ref{th:undec2}.

\subsection{Decomposition of $\nu_{t,(i,j)}$ into Gibbs states}

Heuristically, one expects that at time $t>0$, the local statistics of the dimer configuration
$\eta(t)$ around a point $u\in \mathcal T_L$ will be approximately
that of $\eta$ sampled from a Gibbs state with suitable densities $\nabla_1 \psi(u/L,t),\nabla_2 \psi(u/L,t)$
of dimers of types $1$ and $2$, respectively. Theorem
\ref{th:undec2} below is in a sense a much weaker statement, since it
says only that the (locally) time-space averaged measures $p_{L,t,(i,j)}$ are close to 
linear combinations, with unknown weights, of Gibbs states. This weaker information  is however sufficient for our
purposes (see also Remark \ref{rem:mediegiuste} below).

Recall that $\Omega$ is the set of all perfect matchings of
$\mathcal H$, that we endow with the Borel $\sigma$-algebra generated
by cylindrical sets.  Let $\mathcal E$ be the set of all edges of
$\mathcal H$.  Given $\Lambda\subset \mathcal E$ and $\eta\in\Omega$, we let
$\eta|_\Lambda$ denote the restriction of $\eta$ to $\Lambda$.  A probability
measure $\lambda$ on $\Omega$ is called a \emph{Gibbs measure} if, for every
finite subset $\Lambda\subset\mathcal E$ and for $\lambda$-almost
every dimer configuration $\eta_{\mathcal E\setminus\Lambda }$ on the
edges not in $\Lambda$, the conditional law
$\lambda(\cdot|\eta_{\mathcal E\setminus\Lambda})$ is the uniform law
on the finite set of dimer configurations in $\Lambda$ compatible with
$\eta_{\mathcal E\setminus\Lambda}$ \cite{Sheffield,Georgii}.
We let $\mathcal G_{\mathcal T}$ denote the set of Gibbs measures
that are invariant under the group of translations in $\mathcal T$. 
\begin{Definition}
For every
$\lambda\in \mathcal G_{\mathcal T}$, let
$\hat\rho(\lambda)=(\hat\rho_1(\lambda),\hat \rho_2(\lambda))$, with
$\hat\rho_i(\lambda)=\lambda({\bf 1}_{b_i(0)\in\eta})$ the average
density of dimers of type $i$ under $\lambda$. Clearly, by the
definition of height function,
$\hat\rho(\lambda)\in\mathbb T\cup \partial\mathbb T$.   
\end{Definition}
 Also, let
${\rm ex}\, \mathcal G_{\mathcal T}\subset \mathcal G_{\mathcal T}$ denote the subset of
 Gibbs measures that are ergodic w.r.t.
translations.  If $\rho\in \partial\mathbb T$, then there may exist in
general several Gibbs measures $\lambda \in {\rm ex}\, \mathcal G_{\mathcal T}$ with
$\hat \rho(\lambda)=\rho$. If instead $\rho\in \mathbb T$, it is known
\cite{Sheffield} that there is a unique measure
$\pi_\rho\in {\rm ex}\, \mathcal G_{\mathcal T}$ such that
$\hat\rho(\pi_\rho)=\rho$. In that case, the measure
$\pi_\rho$ can be obtained as the limit as $L\to\infty$ of the uniform measure
on $\Omega_{\bar\rho^{(L)}}$  (provided that
$\bar \rho^{(L)}\to \rho$) and it has a determinantal structure with a rather
explicit kernel
 and power-law decaying correlations \cite{KOS}. 

It is also known that $\mathcal G_{\mathcal T}$ is convex (and actually even a simplex) and that its extreme points are the ergodic measures,  so that the following   decomposition theorem  holds:
\begin{Theorem}(\cite{Georgii}, cf. also \cite[Lemma 3.2.4]{Sheffield})
\label{th:Georgii}
  Given $\nu\in\mathcal G_{\mathcal T}$,
 there exists a unique  $w_\nu\in\mathcal P({\rm ex}\, \mathcal G_{\mathcal T} )$
  (the set of probability measures on
  ${\rm ex}\, \mathcal G_{\mathcal T}$)
such that
\begin{eqnarray}
  \label{eq:undec}
  \nu(\mathrm{d}\eta)=\int_{{\rm ex}\, \mathcal G_{\mathcal T}} \lambda(\mathrm{d}\eta) \, \mathrm{d} w_\nu(\lambda).
\end{eqnarray}

\end{Theorem}
Before proving that the limit measures $\nu_{t,(i,j)}$ have a decomposition of the type \eqref{eq:undec},  we need a preliminary observation:
\begin{Proposition}
\label{prop:kappa}
 Given $\lambda\in {\rm ex}\, \mathcal G_{\mathcal T}$ with $\hat \rho(\lambda)=\rho$ we have: 
  \begin{enumerate}
  \item if $\rho\in\mathbb T$ then
    \begin{eqnarray}
      \label{eq:rho1}
      \lambda(|\epsilon(0,\eta)|k(0,\eta))<\infty.
    \end{eqnarray}
  \item if $\rho\in\partial\mathbb T$ and $\rho_1>0,\rho_2>0$ then 
    \begin{eqnarray}
      \lambda(|\epsilon(0,\eta)|k(0,\eta))=\infty.
    \end{eqnarray}
  \item if $\rho\in\partial\mathbb T$ and  $\min(\rho_1,\rho_2)=0$ then
\begin{eqnarray}
      \lambda(|\epsilon(0,\eta)|k(0,\eta))=0.
    \end{eqnarray}
  \end{enumerate}  
\end{Proposition}
\begin{proof}[Proof of Proposition \ref{prop:kappa}]

  \emph{Claim (1)}. In this case $\lambda=\pi_\rho$, the measure obtained as the $L\to\infty$ limit of the uniform measure on $\Omega_{\rho^{(L)}}$with $\rho^{(L)}\to\rho$. Recall, as discussed
  just after \eqref{eq:ka}, that if $\epsilon(0,\eta)=1$
  (resp. $\epsilon(0,\eta)=-1$) then $k(0,\eta)$ is the smallest $n\ge1$
  such that there is a dimer at the horizontal edge $b_3((n-1) \hat e_3)$ (resp. at
  $b_3 (-n\hat e_3)$).  On the other hand, it is well known (see
  e.g. \cite[Lemma A.1]{T2+1}) that under the measure $\pi_\rho$, the
  distance between two consecutive horizontal dimers  in the same
  vertical column is a random variable with exponential
  tails. Eq. \eqref{eq:rho1} then
  follows. 
  
\emph{Claim (2)}. Since both $\rho_1$ and $\rho_2$ are strictly
positive and $\lambda$ is translation-invariant and ergodic, there is a non-zero
probability that both $b_1(0)$ and $b_2(0)$ belong to $\eta$, in which
case $\epsilon(0,\eta)=1$. On the other hand, since $\rho_1+\rho_2=1$,
there are no dimers of type $3$ and, on the event $\epsilon(0,\eta)=1$, one
has $k(0,\eta)=+\infty$.

\emph{Claim (3)}. Just note that in this case there is almost surely
either no dimer of type $1$ or no dimer of type $2$. Then, from
definition \eqref{eq:epsilon} we see that $\epsilon(0,\eta)=0$ almost
surely.
\end{proof}

The main step in the computation of the r.h.s. of \eqref{eq:inco} will be to show that  any limit point of $p_{L, t, (i,j)}$ admits a decomposition of the type \eqref{eq:undec}:
\begin{Theorem}
\label{th:undec2}
 Let $\nu=\nu_{t,(i,j)}$  be a limit  point  of $\{p_{L,t,(i,j)}\}_L$. There exists a unique
  $w_\nu\in\mathcal P({\rm ex}\, \mathcal G_{\mathcal T} )$
  such that
  \begin{eqnarray}
    \label{eq:ptolimite}
    \nu(d\eta )=\int_{{\rm ex}\, \mathcal G_{\mathcal T}}w_\nu(d\lambda )\lambda(d\eta).
  \end{eqnarray}
Moreover, $w_\nu$ gives mass zero to the subset  \[\{\lambda\in {\rm ex} \mathcal G_{\mathcal T}:  \hat\rho_1(\lambda)>0,\hat\rho_2(\lambda)>0,\hat\rho_1(\lambda)+\hat\rho_2(\lambda)=1 \}.
\]
\end{Theorem}

The next few subsections will be the proof of this theorem. For lightness of notation, we will let $p_L:=p_{L,t,(i,j)}$.
\begin{Remark}\label{rem:noRiesz}
  In \cite[Th. 4.1]{FunakiSpohn}, for the Ginzburg-Landau model, a different
  decomposition theorem was given, for a measure that couples the
  gradients of the height function $H$ and those of the deterministic
  solution of a discretization of the hydrodynamic PDE.  An attempt to adapt
   the rather abstract proof of that result to our model runs into problems  at the step where the Riesz-Markov representation theorem is needed.
The basic reason
  is that, in our case, there are values of $\rho$ (those on the
  boundary of $\mathbb T$) for which more than one ergodic Gibbs
  measure can exist (this phenomenon does not happen for the
  Ginzburg-Landau model).  The ``mesoscopic discretization procedure''
  we devised in Section \ref{sec:discretization}  allows one to avoid
  altogether the use of the coupled measure and also to deal only with finitely many ($N^3$ of them) space-time averaged measures $\nu_{t,(i,j)}$, instead of infinitely many of them as is the case in \cite{FunakiSpohn}.
\end{Remark}

\subsection{Proof of Theorem \ref{th:undec2}}
The proof is divided into various steps. First we show that
$\nu_{t,(i,j)}$ is translation invariant (Section \ref{sub:ti}). Next,
we prove that it has zero entropy production (Section
\ref{sub:ep}). Then, we conclude $\nu_{t,(i,j)}$ is a Gibbs measure
and therefore the decomposition \eqref{eq:ptolimite} holds (Section
\ref{sup:gibbs}). Finally we prove the claim on the support of $w_\nu$
(Section \ref{sub:sup}).
\subsubsection{Translation invariance} \label{sub:ti} The measure $p_L$ is not
translation invariant (it would be if the window $B_j$ in
\eqref{eq:pL} were replaced by the whole torus $\mathcal
T_L$, i.e. if $N=1$). However   since the box $B_j$ is   macroscopic, translation invariance is recovered in the $L\to\infty $ limit:
\begin{Lemma}\label{le:translation_invariance}
Every  limit point $\nu$ of $\{p_L\}_{L\ge1}$ is translation invariant.
\end{Lemma}
\begin{proof}
Let
$f$ be a continuous bounded function on $\Omega$ and $v\in\mathcal T$. One has
\begin{equation}
  |p_L(f)-p_L(f\circ \tau_v)|\le  \|f\|_\infty\frac{|B_j\Delta (\tau_{-v} B_j)|}{|B_j|}
\end{equation}
where $ B_j\Delta (\tau_{-v} B_j)$ is the symmetric difference between $B_j$ and its translate $\tau_{-v} B_j$. Since
$|B_j|=(L/N)^2$ and $|B_j\Delta (\tau_{-v} B_j)|=O(L/N)$, we conclude that $\nu(f)=\nu(f\circ \tau_v)$ for every $v$.
\end{proof}

In order to obtain \eqref{eq:ptolimite}, it is then enough to prove
that $\nu$ is a Gibbs measure and then to apply Theorem
\ref{th:Georgii}.

\subsubsection{Total and local entropy production.}
\label{sub:ep}
 Given a probability distribution $r$ on
$\Omega_{\bar\rho^{(L)}}$, we let
\begin{equation}
  \label{eq:entropia}
  H_L(r|\pi_L)=\sum_{\eta\in\Omega_{\bar\rho^{(L)}}}r(\eta)\log\frac{r(\eta)}{\pi_L(\eta)} \ge 0
\end{equation}
denote the relative entropy of $r$ with respect to $\pi_L$, the
uniform measure on $\Omega_{\bar\rho^{(L)}}$. Also, we let
\begin{multline}
  \label{eq:Itilde}
 \tilde I_L(r):=\frac 1{2L^2}\sum_{\eta\ne\eta'\in \Omega_{\bar\rho^{(L)}}}\pi_L(\eta)\mathcal L_{\eta \eta'}\left[
\frac{r(\eta)}{\pi_L(\eta)}-\frac{r(\eta')}{\pi_L(\eta')}
\right]\\\times\left[\log\frac{r(\eta)}{\pi_L(\eta)}-\log\frac{r(\eta')}{\pi_L(\eta')}  \right]\ge 0
\end{multline}
denote its \emph{entropy production functional}, where
$ \mathcal L_{\eta \eta'}$ is the transition rate from $\eta$ to
$\eta'$ (recall that $\mathcal L_{\eta \eta'}$ is of order $L^2$).
Since $\frac1{L^2}\log|\Omega_{\bar\rho^{(L)}}|$ tends to a positive
constant as $L\to\infty$ \cite{KOS} (the limit is the surface entropy
 $-\sigma(\bar\rho)$, see
\eqref{eq:sigma}) and $\pi_L(\eta)=1/|\Omega_{\bar\rho^{(L)}}|$, one
has the uniform bound
\begin{equation}
\label{eq:bdHL}
  H_L(r|\pi_L)\le \mathcal C L^2 \text{ for every } r\in \mathcal P(\Omega_{\bar\rho^{(L)}}),
\end{equation}
for some $\mathcal C=\mathcal C(\bar\rho)<\infty$.  

The name ``entropy
production'' for $\tilde I_L(r)$ is justified by the fact that, if $r=r_{L,t}$ is the law
of the process at time $t$ with some initial distribution $r_{L,0}$, we have  (using reversibility of
$\pi_L$)
\begin{eqnarray}
  \label{eq:dentr}
  \frac d{dt} H_L(r_{L,t}|\pi_L)=-L^2 \tilde I_L(r_{L,t})\le 0.
\end{eqnarray}
 The entropy production functional $r\mapsto \tilde I_L(r)$ is
 convex and this implies (recalling the definition \eqref{eq:pL} of $p_L=p_{L,t,(i,j)}$)
\begin{multline}
  \label{eq:fs1}
  \tilde I_L(p_L)\le \frac1{|I_i|}\int_{I_i} \mathrm{d}s \frac1{|B_j|}\sum_{u\in B_j} \tilde I_L(r_{L,s}\circ \tau_{-u})=\frac1{|I_i|}\int_{I_i} \mathrm{d}s\tilde I_L(r_{L,s})
\\
=\frac N{tL^2} (H_L(r_{L,(i-1)/N}|\pi_L)-H_L(r_{L,i/N}|\pi_L))\le  \frac{\mathcal CN}t
\end{multline}
where in the first line $\tau_{-u}$ denotes translation by $-u$. In
the first line we used translation invariance of $\pi_L$ and of the
transition rates to write
\[\tilde I_{L}(r_{L,s} \circ \tau_{-u}) = \tilde I_L(r_{L,s}).\]
In the second line we used \eqref{eq:bdHL}, together with $H_L\ge 0$.  If
we define also
\begin{eqnarray}
\label{eq:were}
  I_L(r):=\frac1{2L^2}\sum_{\eta\ne\eta'\in \Omega_{\bar\rho^{(L)}}}\pi_L(\eta)\mathcal L_{\eta \eta'}\left[
\sqrt{\frac{r(\eta)}{\pi_L(\eta)}}-\sqrt{\frac{r(\eta')}{\pi_L(\eta')}}
  \right]^2\\
=\frac1{2L^2}\sum_{\eta\ne\eta'\in \Omega_{\bar\rho^{(L)}}}\mathcal L_{\eta \eta'}\left[
\sqrt{r(\eta)}-\sqrt{r(\eta')}
  \right]^2,
\end{eqnarray}
using the inequality
\begin{eqnarray}
  2(\sqrt u-\sqrt v)^2\le (u-v)(\log u-\log v), \; u,v>0
\end{eqnarray}
(just write $\sqrt u-\sqrt v=\int_v^u dt/(2\sqrt t)$ and use Cauchy-Schwarz)
we deduce:
\begin{Lemma}\label{le:prod_tore} Letting  $\mathcal{C} = \mathcal{C}(\bar \rho) < \infty$ denote the same constant as in \eqref{eq:bdHL}, we have 
\begin{eqnarray}
\label{eq:ct}
   I_L(p_L)\le\frac {\mathcal C N}t.
\end{eqnarray}
\end{Lemma}

Lemma \ref{le:prod_tore} states that the total entropy production per unit time of $p_L$ is bounded independently of $L$. The next step is to use the fact that this is an extensive quantity to deduce that the entropy production of the limit measure $\nu$ in any finite window is $0$.

Given a finite subset $\Lambda\subset \mathcal E$ of edges of
$\mathcal H$, let $\partial\Lambda$ denote the set of edges in
$\mathcal E\setminus \Lambda$ that are incident to at least one edge in
$\Lambda$. Given a dimer configuration $\eta$, we denote
$\eta_\Lambda$ and $\eta_{\partial \Lambda}$ its restriction to
$\Lambda$ and $\partial \Lambda$, respectively and call
$\Omega(\eta_{\partial\Lambda})$ the set of configurations
$\eta_\Lambda$ compatible with $\eta_{\partial\Lambda}$.  Let
$\mathcal L^{\eta_{\partial\Lambda}}$ be the generator of the
restricted dynamics where only updates that do not move dimers on
edges in $\mathcal E\setminus \Lambda$ are allowed.
\begin{Remark}
As remarked in Section \ref{sec:thedyn}, any particle jump by $\pm m \hat e_3$ can be seen as the concatenation of 
$m$ elementary rotations,  in $m$ vertically stacked adjacent hexagonal faces of $\mathcal H_L$, of three dimers. Then, more explicitly, allowed moves of the restricted dynamics with generator $\mathcal L^{\eta_{\partial\Lambda}}$ are only those particle jumps such that none of the $m$ elementary dimer rotations changes the dimer occupation variables at edges outside $\Lambda$.
\end{Remark}
 Of course,
$\mathcal L^{\eta_{\partial\Lambda}}_{\eta,\eta'}$ is non-zero only if
$\eta$ and $\eta'$ coincide outside of $\Lambda$. With some
abuse of notation we will sometimes write
$\mathcal L^{\eta_{\partial\Lambda}}_{\eta_\Lambda,\eta'_\Lambda}$
instead of $\mathcal L^{\eta_{\partial\Lambda}}_{\eta,\eta'}$. In
fact, the matrix elements of the generator are independent of
$\eta_{(\Lambda\cup\partial\Lambda)^c}$ and
$\eta'_{(\Lambda\cup\partial\Lambda)^c}$.  
Given a probability measure $r$ on
$\Omega$, define
\begin{eqnarray}
  \label{eq:Ilambda}
  I_\Lambda(r)=\frac1{2L^2}\int_\Omega\pi_{\bar \rho}(d\eta)\sum_{\eta'\in\Omega} \mathcal L^{\eta_{\partial\Lambda}}_{\eta_\Lambda,\eta'_\Lambda}\left(
  \sqrt{\frac{r(\eta_\Lambda,\eta_{\partial\Lambda})}{\pi_{\bar \rho}(\eta_\Lambda,\eta_{\partial\Lambda})}}-\sqrt{\frac{r(\eta'_\Lambda,\eta_{\partial\Lambda})}{\pi_{\bar \rho}(\eta'_\Lambda,\eta_{\partial\Lambda})}}
  \right)^2
\end{eqnarray}
where
$r(\eta_\Lambda,\eta_{\partial\Lambda})$ is the probability under $r$ that the dimer configuration restricted to $\Lambda$ and $\partial\Lambda$ is $\eta_\Lambda,\eta_{\partial\Lambda}$, respectively, 
and similarly for
$\pi_{\bar \rho}(\eta_\Lambda,\eta_{\partial\Lambda})$. Using
 the symmetry
$\mathcal
L^{\eta_{\partial\Lambda}}_{\eta_\Lambda,\eta'_\Lambda}=\mathcal
L^{\eta_{\partial\Lambda}}_{\eta'_\Lambda,\eta_\Lambda}$ and 
in particular that
$\pi_{\bar \rho}(\eta_\Lambda,\eta_{\partial\Lambda})=\pi_{\bar \rho}(\eta'_\Lambda,\eta_{\partial\Lambda})$
whenever $\eta$ and $\eta'$ are obtained one from the other via an
update of the restricted dynamics, one can rewrite $I_\Lambda(r)$ as
the finite sum
\begin{eqnarray}
  \label{eq:Ilambda2}
  I_\Lambda(r)=\frac1{2L^2} \sum_{\eta_{\partial\Lambda}}\sum_{\eta_\Lambda,\eta'_\Lambda\in\Omega(\eta_{\partial \Lambda})}\mathcal
L^{\eta_{\partial\Lambda}}_{\eta_\Lambda,\eta'_\Lambda}\left( \sqrt{r(\eta_\Lambda,\eta_{\partial\Lambda})}-\sqrt{r(\eta'_\Lambda,\eta_{\partial\Lambda})}
  \right)^2.
\end{eqnarray}

\begin{Lemma}\label{le:prod_locale}
Let $\nu$ be any limit point of $\{p_L\}_{L\ge1}$. Then, for any finite $\Lambda$,
\begin{equation}\label{eq:prod_locale}
I_\Lambda(\nu) = 0.
\end{equation}
\end{Lemma}
\begin{proof}
In Eq. \eqref{eq:Ilambda2} let us take in particular   $r=p_L,$ that  is concentrated on the $L$-periodic dimer configurations in $\Omega_{\bar\rho^{(L)}}$. 
In this case, 
using the inequality 
\[
\left(\sqrt{\sum_i a_i}-\sqrt{\sum_i  b_i}\right)^2\le \sum_i (\sqrt{a_i}-\sqrt{b_i})^2, \quad 
a_i,b_i\ge0,
\]
(that is just  Cauchy-Schwarz) we see that 
\begin{eqnarray}
  \label{eq:CS}
   I_\Lambda(p_L)\le 
\frac1{2L^2}
\sum_{\eta,\eta'\in \Omega_{\bar\rho^{(L)}}}\mathcal
L^{\eta_{\partial\Lambda}}_{\eta_\Lambda,\eta'_\Lambda}\left( \sqrt{p_L(\eta)}-\sqrt{p_L(\eta')}
  \right)^2.
\end{eqnarray}
To prove the claim of the Lemma, let $V_n=\{1,\dots,n\}^2$ and for
$v\in V_n$ let $\Lambda^v:=\tau_v\Lambda$ denote the $v-$translation of
$\Lambda$.
We write
\begin{multline}
  \frac 1{n^2}\sum_{v\in V_n}I_{\Lambda^v}(p_L)\le \frac1{n^2}\frac1{2L^2}\sum_{\eta,\eta'\in \Omega_{\bar\rho^{(L)}}}\sum_{v\in V_n}
\mathcal
L^{\eta_{\partial\Lambda^v}}_{\eta,\eta'}\left( \sqrt{p_L(\eta)}-\sqrt{p_L(\eta')}\right)^2\\
\le \frac{f(\Lambda)}{n^2}\frac1{2L^2}\sum_{\eta,\eta'\in \Omega_{\bar\rho^{(L)}}}\mathcal
L_{\eta,\eta'}\left( \sqrt{p_L(\eta)}-\sqrt{p_L(\eta')}\right)^2\\=\frac{f(\Lambda)}{n^2} I_L(p_L)\le  \frac{f(\Lambda)}{n^2} 
\frac{\mathcal CN}t.
\end{multline}
The first inequality is just \eqref{eq:CS}; the second is obtained by
remarking that the set of transitions $\eta\mapsto \eta'$ allowed by
$\mathcal L^{\eta_{\partial\Lambda^v}}$ for some $v\in V_n$ is
contained in the set of transitions allowed by the full generator
$\mathcal L$.  When taking the sum over $v\in V_n$ some transitions
may be counted more than once (because the sets $\Lambda^v$ are not
disjoint) but this multiplicity is bounded by some finite $f(\Lambda)$
independent of $n$. In the third line, we recognized the definition
\eqref{eq:were} of $I_L(p_L)$ and then we used inequality
\eqref{eq:ct} in the last step.

From \eqref{eq:Ilambda2} we see that $r\mapsto I_\Lambda(r)$ is
continuous and since by assumption $p_{L} $ converges  weakly (along some sub-sequence $\{L_m\}_{m\ge1}$) to $\nu$ that is
translation invariant, we conclude that
\begin{equation}
  I_\Lambda(\nu)= \frac 1{n^2}\sum_{v\in V_n}I_{\Lambda^v}(\nu)=\frac 1{n^2}\lim_{m\to\infty} \sum_{v\in V_n}I_{\Lambda^v}(p_{L_m})\le \frac{f(\Lambda)}{n^2} 
\frac{\mathcal CN}t
\end{equation}
so that, taking $n\to\infty$, we conclude that $I_\Lambda(\nu)=0$ for every finite $\Lambda$.
\end{proof}

\subsubsection{Gibbs property}
\label{sup:gibbs}
The next step in the proof of Theorem \ref{th:undec2} is the following:
\begin{Lemma}\label{le:claiGibbs} 
Let $\nu$ be a probability measure on $\Omega$ such that $I_\Lambda(\nu) = 0$ for every finite $\Lambda$. Then $\nu$ is a Gibbs measure.
\end{Lemma}   

\begin{proof}
 Recall that, by definition, the statement that $\nu$ is a Gibbs measure means
 that, for every finite $\Lambda$ and for $\nu$-almost every
 $\eta_{\mathcal E\setminus \Lambda}$, the measure
 $\nu(\cdot|\eta_{\mathcal E\setminus \Lambda})$ is uniform on
 $\Omega(\eta_{\partial\Lambda})$.  If the restricted dynamics with
 generator $\mathcal L^{\eta_{\partial \Lambda}}$ were ergodic for
 every $\Lambda$ and $\eta_{\partial\Lambda}$ then from the vanishing of
 \eqref{eq:Ilambda2} we would deduce that
 $\nu(\eta_\Lambda,\eta_{\partial\Lambda})$ is constant
 w.r.t. $\eta_\Lambda\in\Omega(\eta_{\partial\Lambda})$ and the claim would
 be easy to conclude. However, ergodicity may fail to be satisfied
 for some choice of $\Lambda,\eta_{\partial\Lambda}$, see e.g. Fig. \ref{fig:nonergodico},
 so the argument requires some care.

\begin{figure}[h]
  \includegraphics[width=9cm]{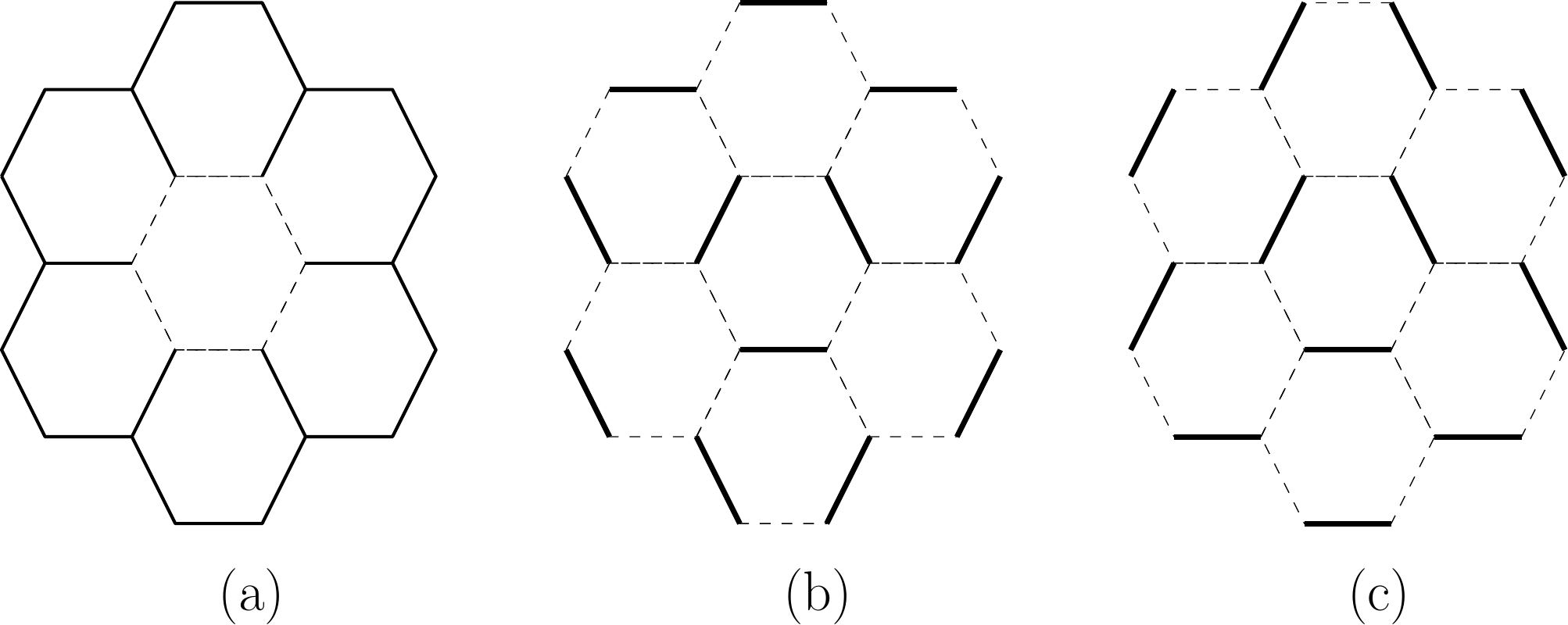}
  \caption{(a) the domain $\Lambda$ (dashed edges are not
    included). Configurations (b) and (c) coincide outside $\Lambda$
    but it is not possible to go from one to the other without moving
    dimers on the sides of the central hexagon (which are not in $\Lambda$).}
\label{fig:nonergodico}
\end{figure}

To circumvent this difficulty, we observe first of all the following
(see proof below):
\begin{Claim}
  \label{claim:ergodico}
  Let $\Lambda_n$ be the set of edges of the collection of hexagonal faces of $\mathcal H$ with label $u=(u_1,u_2), -n\le u_i\le n$. For every configuration $\eta\in\Omega$, the restricted dynamics in $\Lambda_n$,
  with generator $\mathcal L^{\eta_{\partial\Lambda_n}}$, is
  ergodic. In particular, from the assumption $I_{\Lambda_n}(\nu) = 0$ and \eqref{eq:Ilambda2} 
  we see that the marginal of $\nu(\cdot|\eta_{\partial\Lambda_n})$ on $\eta_{\Lambda_n}$  is
  uniform on $\Omega(\eta_{\partial\Lambda_n})$.
\end{Claim}
(The important point is that $\Lambda_n$ covers the whole of $\mathcal H$ as $n\to\infty$ and that, in contrast with the domain $\Lambda$ in Fig. \ref{fig:nonergodico},   it contains all the edges of all the faces in its interior).

Given Claim \ref{claim:ergodico},  we have also that for $m>n$, 
the marginal on $\eta_{\Lambda_n}$ of  $\nu(\cdot|\eta_{(\Lambda_m\setminus\Lambda_n)\cup \partial \Lambda_m})$   is uniform  on  $\Omega(\partial\Lambda_n)$,
simply because $\Lambda_m\supset\Lambda_n$ and conditioning the uniform measure
$\nu(\cdot|\eta_{\partial\Lambda_m})$ on
$\Omega(\eta_{\partial \Lambda_m})$ to the value of $\eta$ on
$(\Lambda_m\setminus \Lambda_n)$ gives
again a uniform measure. Taking the limit $m\to\infty$ implies that
the measure
\[
\nu(\cdot|\eta_{\mathcal E\setminus\Lambda_n})
\]
is uniform on $\Omega(\eta_{\partial\Lambda_n})$.

Given a general finite set of edges $\Lambda$, take $n$ large enough so that $\Lambda\subset\Lambda_n$ and 
write 
\begin{eqnarray}
  \nu(\cdot|\eta_{\mathcal E\setminus\Lambda})=\tilde\nu_n^\eta(\cdot|\eta_{\Lambda_n\setminus \Lambda}), 
\;\tilde \nu_n^\eta(\cdot):=\nu(\cdot|\eta_{\mathcal E\setminus \Lambda_n}).
\end{eqnarray}
Uniformity of $\tilde\nu_n^\eta$ on $\Omega(\eta_{\partial\Lambda_n})$ implies uniformity of $\nu(\cdot|\eta_{\mathcal E\setminus\Lambda})$ on $\Omega(\eta_{\partial \Lambda})$. 
\end{proof}

\begin{proof}[Proof of Claim \ref{claim:ergodico}]
  Let $\eta\ne\eta'$ be two configurations in
  $\Omega(\eta_{\partial\Lambda_n})$: their height functions $H_\eta,H_{\eta'}$
  coincide on the faces outside $\Lambda_n$ and differ on
  at least a face $u$ in $\Lambda_n$ (say that
  $H_\eta(u)<H_{\eta'}(u)$). Starting from $u$, consider a
  nearest-neighbor path $\mathcal C=\{u=u_0,u_1,\dots\}$ on 
  faces, that moves only in directions $+\hat e_1,+\hat e_2$ or
  $+\hat e_3$ and such that the edge crossed from $u_i$ to $u_{i+1}$
  is \emph{not} occupied by a dimer in configuration $\eta$. Observe
  that, from Definition \ref{def:height} of the height function, the
  height difference $H_{\eta'}-H_\eta$ is non-decreasing along the
  path. As a consequence, $\mathcal C$ cannot reach a face outside
  $\Lambda_n$, where the heights coincide. Also, $\mathcal C$ cannot
  form a closed loop $u_0,\dots, u_n=u_0$. In fact, since none of the
  edges crossed by the path being  occupied by dimers and the
  densities $\bar\rho_i^{(L)}$ are strictly positive, we would find
  from \eqref{eq:gradh} that $H_\eta(u_n)<H_\eta(u_0)=H_\eta(u_n)$. As a
  consequence, any such path $\mathcal C$ must stop at some face $u'$
  inside $\Lambda_n$: since $\mathcal C$ cannot be continued, we see
  that necessarily the three edges $b_1(u'),b_2(u'),b_3(u')$ are all
  occupied by a dimer in $\eta$. Then, we can rotate these three dimers around
  face $u'$, and the effect is that $H_\eta(u')$ increases by $+1$
  (the rotation is a legal move of the restricted dynamics, since all
  edges of $u$ are in $\Lambda_n$). The mutual volume
  $\sum_u|H_\eta(u)-H_{\eta'}(u)|$ decreases by $1$ because, as we
  remarked above,
  \[H_{\eta'}(u')-H_\eta(u')\ge H_{\eta'}(u)-H_\eta(u)>0.\] We
  iterate the procedure as long as there exist faces where
  $H_{\eta'}(u)\ne H_\eta(u)$. When the procedure stops and the mutual volume is zero,
  we have obtained a chain of legal moves that leads from $\eta$ to
  $\eta'$ and therefore the  dynamics is ergodic.
\end{proof}

Now the proof of the decomposition stated in Theorem \ref{th:undec2}
is just a matter of putting together Lemmas
\ref{le:translation_invariance}, \ref{le:prod_locale} and
\ref{le:claiGibbs} to see that $\nu$ is a translation invariant Gibbs
measure and then applying Theorem \ref{th:Georgii}.

\subsubsection{Support of $w_\nu$}\label{sub:sup}
Recall that we are writing for ease of notation $\nu:=\nu_{t,(i,j)}$.
We have proven  that \eqref{eq:ptolimite} holds and it remains to show that $w_\nu$ gives zero mass to the set
\begin{eqnarray}
  \label{eq:zm}
\{\lambda\in  {\rm ex} \mathcal G_{\mathcal T}:  \hat\rho_1(\lambda)>0,\hat\rho_2(\lambda)>0,\hat\rho_1(\lambda)+\hat\rho_2(\lambda)=1 \}.  
\end{eqnarray} 
In fact,
from Proposition \ref{prop:contrk} we know that
\begin{eqnarray}
 \int_{{\rm ex} \mathcal G_{\mathcal T}}w_\nu(d\lambda)\lambda(|\epsilon(0,\eta)|k(0,\eta))= \nu(|\epsilon(0,\eta)|k(0,\eta))\le N^3 K(t).
\end{eqnarray}
On the other hand, from point (3) of Proposition \ref{prop:kappa} we
know that \[\lambda(|\epsilon(0,\eta)|k(0,\eta))=+\infty\]
for every $\lambda\in {\rm ex}\, \mathcal G_{\mathcal T}$ satisfying
\eqref{eq:zm}. Then, the claim follows. 

\section{Conclusion of the proof of Theorem \ref{th:main}}
\label{sec:calcolo}

We have shown in the previous section that each limiting measure $\nu_{t,(i,j)}$ is a combination of Gibbs measures. This allows us to show that the r.h.s. of \eqref{eq:inco} is asymptotically smaller than  or equal to zero,
 i.e. the $\mathbb L^2$ distance between $\psi$ and $H/L$ stays small at all times:
\begin{Theorem}\label{th:IL}  For every $M>0,N\ge1,(i,j)\in\mathcal I_N$ and $t>0$ there exist real functions $r^{(i,j)}_{M,N}(t)$ and $g_N^{(i,j)}(t)$
satisfying 
\begin{eqnarray}
\label{eq:errpiccoli}
\sup_N\frac1{N^3}  \sum_{(i,j)\in\mathcal I_N} r^{(i,j)}_{M,N}(t)\stackrel {M\to\infty}\to0\\
\label{eq:erp2}
\limsup_{N\to\infty}\frac1{N^3}  \sum_{(i,j)\in\mathcal I_N} g_N^{(i,j)}(t)=0 
\end{eqnarray}
such that 
\begin{eqnarray}
\label{gdt}
 \nu_{t,(i,j)} (C^{(i,j)}(\eta)+U^{(i,j)}_M(\eta))\le g_N^{(i,j)}(t)+ r_{M,N}^{(i,j)}(t).
\end{eqnarray}
\end{Theorem}
As a consequence, plugging \eqref{gdt} into \eqref{eq:inco} and  taking first $N\to\infty$ and then $M\to\infty$ we deduce
\begin{eqnarray}
\label{eq:ducr}
  \limsup_{m\to\infty}D_{L_m}(t)=0.
\end{eqnarray}
Actually, our argument up to here yields that for every  sub-sequence
$\{\tilde L_k\}_{k\ge1}$ it is possible to extract a sub-sequence
$\{L_m\}_{m\ge1}\subset \{\tilde L_k\}_{k\ge 1}$ such that
\eqref{eq:ducr} holds.  Then, it follows immediately that
\eqref{eq:ducr} holds along \emph{any} sub-sequence increasing to $+\infty$, and therefore
claim (ii) of Theorem \ref{th:main} is proven.  The rest of the
present section is devoted to the proof of Theorem \ref{th:IL}.

\begin{proof}[Proof of Theorem \ref{th:IL}]

The main point of this proof is that we can explicitly compute (apart
from a few technicalities stemming from the cut-off $M$) the
expectation of the observables $C^{(i,j)}(\eta)$ and
$U^{(i,j)}_M(\eta)$ when $\eta$ is sampled from  an ergodic
Gibbs measure. From the decomposition of $\nu_{t,(i,j)}$ into ergodic
Gibbs measures provided by Theorem \ref{th:undec2}, this allows us to
write the l.h.s. of \eqref{gdt} as the integral  over ${\rm ex\,}\mathcal G_{\mathcal T}$ of the
density $w_{\nu_{t,(i,j)}}(d\lambda)$ times an explicit function of
$\lambda$ (the function in the r.h.s. of \eqref{eq:ceccone}, with $G$ defined in \eqref{eq:Gg}).  At that
point, a non-trivial algebraic identity implies that the integral is
 (asymptotically in the limit $\lim_{M\to\infty}\lim_{N\to\infty}$)
non-positive independently of the \emph{unknown} density $w_{\nu_{t,(i,j)}}$,
which concludes the proof. This point is related to the fact that the
limit PDE \eqref{eq:PDE} contracts the $\mathbb L^2$ distance between
solutions, a fact which was already pointed out in \cite{LThydro}.

\medskip

Using \eqref{450}, the l.h.s. of \eqref{gdt} equals
\begin{multline}
  \label{eq:pp}
  \nu_{t,(i,j)}\left[C^{(i,j)}(\eta)-
  (\bar\rho_3-z_1^{(i,j)}-z_2^{(i,j)})|\epsilon(0,\eta)|\frac{k(0,\eta)-1}2\right]\\
+\nu_{t,(i,j)}\left[  (\bar\rho_3-z_1^{(i,j)}-z_2^{(i,j)})|\epsilon(0,\eta)|\frac{k(0,\eta)-1}2{\bf 1}_{|\epsilon(0,\eta)|k(0,\eta)> M}\right]\\
\le \nu_{t,(i,j)}\left[C^{(i,j)}(\eta)-
  (\bar\rho_3-z_1^{(i,j)}-z_2^{(i,j)})|\epsilon(0,\eta)|\frac{k(0,\eta)-1}2\right]\\
+c\nu_{t,(i,j)}\left[ |\epsilon(0,\eta)|k(0,\eta){\bf 1}_{|\epsilon(0,\eta)|k(0,\eta)> M}\right]
\end{multline}
for some absolute constant $c<\infty$.
Letting
\begin{eqnarray}
  r_{M,N}^{(i,j)}(t):=c\nu_{t,(i,j)}\left[ |\epsilon(0,\eta)|k(0,\eta){\bf 1}_{|\epsilon(0,\eta)|k(0,\eta)> M}\right]
\end{eqnarray}
we see that 
\begin{eqnarray}
\label{eq:esva}
  \frac1{N^3}\sum_{(i,j)\in\mathcal I_N} r_{M,N}^{(i,j)}(t)=c \mu_t\left[|\epsilon(0,\eta)|k(0,\eta){\bf 1}_{|\epsilon(0,\eta)|k(0,\eta)> M}\right]
\end{eqnarray}
where
 \[
\mu_t(f(\eta)):=\frac1{N^3}\sum_{(i,j)\in\mathcal I_N}\nu_{t,(i,j)}(f(\eta)) =
 \lim_{m\to\infty}\frac1{t L_m^2}\sum_{u\in\mathcal T_{L_m}}\int_0^t ds \mathbb E (f(\eta_{-u}(s)))
.
\]
Note in particular that the r.h.s. of \eqref{eq:esva} is independent of $N$.
Then,  \eqref{eq:errpiccoli} follows from \eqref{eq:nuK} of 
 Proposition \ref{prop:contrk} and dominated convergence.
\smallskip

Next, we claim (this is proved at the end of the section):
\begin{Proposition} 
\label{prop:ceccone}
For every $\lambda$ in the support of $w_{\nu_{t,(i,j)}}(d\lambda)$,
\begin{multline}
  \label{eq:ceccone}
 \lambda\left[C^{(i,j)}(\eta)-\left(\bar\rho_3-z^{(i,j)}_1-z^{(i,j)}_2\right)|\epsilon(0,\eta)|
\frac{k(0,\eta)-1}2\right]\\=G(\hat\rho(\lambda),\bar\rho,z^{(i,j)})+\frac12\left[-z^{(i,j)}\cdot \bar\rho+\bar \rho_1(\hat\rho_1(\lambda)-\bar\rho_1)+\bar\rho_2(\hat\rho_2(\lambda)-\bar\rho_2)\right]
\end{multline}
where 
$G\le 0$ is an explicit function that is defined in Eqs. \eqref{eq:Gg} and \eqref{G2} below.
\end{Proposition}

We deduce that
\begin{multline}
  \nu_{t,(i,j)}(C^{(i,j)}(\eta)+U_M^{(i,j)}(\eta))\le r^{(i,j)}_{M,N}(t)
+g_N^{(i,j)}(t),
\\
g_N^{(i,j)}(t):=\frac12\int_{{\rm ex}\,\mathcal G_{\mathcal T}} w_{\nu_{t,(i,j)}}(d\lambda)\left[-z^{(i,j)}\cdot \bar\rho+\bar \rho_1(\hat\rho_1(\lambda)-\bar\rho_1)+\bar\rho_2(\hat\rho_2(\lambda)-\bar\rho_2)\right]
\end{multline}
and it remains to prove \eqref{eq:erp2} to conclude the proof of the Theorem.
First of all, using the definition \eqref{eq:zij} of $z^{(i,j)}$ and
the smoothness of  $\psi(u,t)$ in space and time,
\begin{eqnarray}\nonumber
  \frac1{N^3}\sum_{(i,j)\in\mathcal I_N}z^{(i,j)} := \frac1{N^3}\sum_{(i,j)\in\mathcal I_N}\nabla\psi(j/N,ti/N)=\frac1t \int_0^t \int_{[0,1]^2}\nabla\psi(u,s)+\epsilon_N:
\end{eqnarray}
 the integral is zero because $\psi$ is periodic on the torus while $\epsilon_N$ tends to zero as $N\to\infty$.
 
Similarly, for $a=1,2$,
\begin{multline}
\frac1{N^3}\sum_{(i,j)\in\mathcal I_N}    \int_{ {\rm ex}\, \mathcal G_{\mathcal T}}w_{\nu_{t,(i,j)}}(d\lambda) (\hat\rho_a(\lambda)-\bar\rho_a)=\frac1{N^3}\sum_{(i,j)\in\mathcal I_N}    \nu_{t,(i,j)}({\bf 1}_{b_a(0)\in\eta}-\bar\rho_a)\\
=\lim_{m\to\infty}\frac1{L_m^2}\frac1t \int_0^tds \left[\sum_{u\in\mathcal T_{L_m}} \mathbb P(b_a(u)\in\eta(s))-L_m^2\bar\rho^{(L_m)}_a\right]
\end{multline}
and the last expression is zero because any configuration
$\eta\in\Omega_{\bar\rho^{(L)}}$ contains exactly
$L^2 \bar\rho^{(L)}_a$ dimers of type $a$.
This concludes the proof of Theorem  \ref{th:IL}.
\end{proof}
\begin{proof}[Proof of Proposition \ref{prop:ceccone}]
Going back to the definition \eqref{eq:C1} of $C_1,\dots,C_4$ and to the definition of $C^{(i,j)}$ in Section 
\ref{sec:discretization}, we see that
\begin{multline}
  \label{eq:pallosa}
  C^{(i,j)}(\eta)=\frac{|\epsilon(0,\eta)|}2+\frac{\bar\rho_3}2F(0,\eta)
  \\-\frac1{2}\left(
  {\bf 1}_{b_1(0)\in\eta}(1-2\bar\rho_1)+\bar\rho_1^2+{\bf 1}_{b_2(0)\in\eta}(1-2\bar\rho_2)+\bar\rho_2^2
  \right)\\
  -
  \frac{1}2F(0,\eta)(z^{(i,j)}_1+z^{(i,j)}_2)
  +\frac1{2} \left(({\bf 1}_{b_1(0)\in\eta}-\bar\rho_1)z^{(i,j)}_1+({\bf 1}_{b_2(0)\in\eta}-\bar\rho_2)z^{(i,j)}_2
  \right)
  \\
  + 2W(z^{(i,j)}+\bar \rho)\cdot({\bf 1}_{b_1(0)\in\eta}-\bar\rho_1,{\bf 1}_{b_2(0)\in\eta}-\bar\rho_2)-
2z^{(i,j)}\cdot W(z^{(i,j)}+\bar \rho).
\end{multline}
Therefore, 
recalling also the definition \eqref{eq:Ff} of $F$, we see that
to compute the l.h.s. of \eqref{eq:ceccone}, we need to compute the average (w.r.t. $\lambda$) of 
the following functions:
\[
{\bf 1}_{b_1(0)\in\eta}, \;\; {\bf 1}_{b_2(0)\in\eta}, \;\; {\bf 1}_{\{b_1(0),b_2(0)\}\subset\eta}, \;\;|\epsilon(0,\eta)| \;\; \text{and}
\; \;|\epsilon(0,\eta)|(k(0,\eta)-1).
\]
Of course one has (by definition)
\begin{eqnarray}
\label{eq:ecomeno}
  \lambda(b_i(0)\in\eta)=\hat\rho_i(\lambda), i=1,2.
\end{eqnarray}
The remaining averages are not at all as trivial.

\emph{Case 1: $\lambda$ is such that
  $\rho:=\hat \rho(\lambda)\in\mathbb T$}. As we already mentioned,
there is a unique such translation invariant, ergodic Gibbs measure,
that we denote $\pi_\rho$. The  determinantal structure
\cite{KOS} of $\pi_\rho$ allows for several explicit
computations. Notably, the following identities hold: if
$\lambda=\pi_\rho$, then
  \begin{gather}
\label{LT1}
    \lambda(\{b_1(0),b_2(0)\}\subset\eta)=\rho_1\rho_2+(1-\rho_1-\rho_2)V(\rho)
\\
\label{LT2}
\lambda\left[|\epsilon(0,\eta)|\frac{k(0,\eta)-1}2\right]=-\rho_1\rho_2+(\rho_1+\rho_2)V(\rho),
  \end{gather}
where 
\begin{eqnarray}
  \label{eq:Vv}
  V(\rho)=\frac1\pi\frac{\sin(\pi\rho_1)\sin(\pi\rho_2)}{\sin(\pi(1-\rho_1-\rho_2))}.
\end{eqnarray}
Equations \eqref{LT1} and \eqref{LT2} are proven in
\cite[Prop. 14]{LThydro}, as a consequence of a combinatorial identity
discovered in \cite{CF}. The indicator ${\bf 1}_{X(0,0)}$ that appears
in \cite[Eq. (3.33), (3.34)]{LThydro} is nothing but
$|\epsilon(0,\eta)|$, while $n(b^{(0,0)})$ there is what we call $k(0,\eta)$
in the present work.

Moreover,
\begin{multline}
  \label{eq:nnpp}
  \lambda(|\epsilon(0,\eta)|)=\pi_\rho(\{b_1(0),b_2(0)\}\subset\eta)+\pi_\rho(\{b_1(-\hat
  e_1),b_2(-\hat
  e_2)\}\subset\eta)\\=2\pi_\rho(\{b_1(0),b_2(0)\}\subset\eta).
\end{multline}
The last equality follows from the fact that the measure $\pi_\rho$ is
invariant under reflection through the vertex $0$ (which maps $b_i(0)$
into $b_i(-\hat e_i),i=1,2$): this holds because the reflected measure is still Gibbs and
has the same dimer densities as $\pi_\rho$, and we mentioned that for
$\rho\in\mathbb T$ there is a unique ergodic Gibbs measure with
$\hat\rho(\lambda)=\rho$.  Given
Eqs. \eqref{eq:ecomeno}-\eqref{eq:nnpp}, it is not hard to check that the
l.h.s. of \eqref{eq:ceccone}  equals the
r.h.s. provided that
\begin{eqnarray}
  \label{eq:Gg}
G(\rho,\bar\rho,z)=-2(W(z+\bar\rho)-W(\rho))\cdot(z+\bar\rho-\rho).  
\end{eqnarray}
Now we claim  that
\begin{eqnarray}
  \label{eq:a-b}
   (W(a)-W(b))\cdot(a-b)\ge0 \;  \text{ for }  a,b\in  {\mathbb T}
\end{eqnarray}
and actually that equality holds only if $a=b$.
To prove \eqref{eq:a-b}, it is sufficient to prove that
\begin{eqnarray}
  \label{eq:giusta}
  v\cdot H_W(\rho)v\ge0 
\end{eqnarray}
for every $v\in \mathbb R^2$ and  for every $\rho\in  {\mathbb T}$, where $H_W(\rho)$ is 
the matrix
\begin{eqnarray}
  \label{eq:24aa}
  H_W(\rho):=\left(
  \begin{array}{cc}
    \partial_{\rho_1} W_1(\rho) & \partial_{\rho_2} W_1(\rho)\\
\partial_{\rho_1} W_2(\rho) & \partial_{\rho_2} W_2(\rho)
  \end{array}
\right).
\end{eqnarray}
In turn, this follows if the symmetric matrix $\tilde H_W(\rho)=1/2(H_W(\rho)+H_W(\rho)^T)$
is positive definite.
An explicit computation shows that 
the trace of $\tilde H_W$ is
\begin{eqnarray}
  \label{eq:25aa}
  {\rm Tr}(\tilde H_W(\rho))=\frac12 \frac{\sin^2(\pi \rho_1)+\sin^2(\pi \rho_2)}{\sin^2(\pi(\rho_1+\rho_2))} > 0
\end{eqnarray}
and
\begin{eqnarray}
  \label{eq:26aa}
  \det (\tilde H_W(\rho))=\frac{\sin^2(\pi \rho_1)\sin^2(\pi \rho_2)}{4\sin^2(\pi(\rho_1+\rho_2))} >0.
\end{eqnarray}
This concludes the proof of \eqref{eq:giusta} and therefore of the claim in the case $\rho:=\hat \rho(\lambda)\in\mathbb T$.

\emph{Case 2: $\lambda$ is such that $\rho=\hat\rho(\lambda)$ satisfies $\min(\rho_1,\rho_2)=0$}.  Assume  w.l.o.g. that
$\rho_2=0$.  In this case, as already observed in Proposition
\ref{prop:kappa}, $|\epsilon(0,\eta)|$ is $\lambda$-almost surely zero
and the l.h.s. of \eqref{LT1}, \eqref{LT2} and \eqref{eq:nnpp} equal
$0$.  Then, simple algebra shows that  \eqref{eq:ceccone} holds with
\begin{eqnarray} 
\label{G2}
G((\rho_1,0),\bar\rho,z)=-2 W(z+\bar \rho)\cdot(z+\bar \rho-(\rho_1,0))-\frac12(\bar \rho_2+z_2)\rho_1.  
\end{eqnarray}
Next, we remark that the r.h.s. of \eqref{G2}
equals
\begin{eqnarray}
  \label{eq:Glim}
  \lim_{\mathbb T\ni x\to(\rho_1,0)} -2(W(z+\bar \rho)-W(x))\cdot (z+\bar \rho-x).
\end{eqnarray}
To be precise, in the case $\rho_1=1$ this is true provided
$x=(x_1,x_2)$ tends to $(\rho_1,0)$ in such a way that
$x_2=o(\rho_1-x_1)$ (for $\rho_1<1$ this is not necessary, since $W$
can be continuously extended from $\mathbb T$ to the subset of
$\partial\mathbb T$ where $\rho_1+\rho_2<1$).  Then, the negativity of
\eqref{G2} follows from \eqref{eq:a-b}.

In both cases we have shown that \eqref{eq:ceccone} holds with $G\le 0$ and Proposition \ref{prop:ceccone} is proven.
\end{proof}
\begin{Remark}
  \label{rem:mediegiuste}
  A posteriori, one can argue that the measure $w_{\nu_{t,(i,j)}}$ in
  Theorem \ref{th:undec2} is asymptotically (as $N\to\infty$)
  concentrated on the Gibbs measure $\pi_\rho$ with
  $\rho=\bar \rho+z^{(i,j)}$. We do not wish to formulate this as a
  precise theorem and prefer to give the intuitive reason instead. For
  $\rho\ne \bar \rho+z$ the function $G$ in \eqref{eq:Gg} is strictly
  negative and, if the measure $w_{\nu_{t,(i,j)}}$ gave positive mass
  to the ``wrong'' densities $\rho$, one would conclude from
  \eqref{eq:inco} and Proposition \ref{prop:ceccone} that
  $\limsup_L D_L(t)$ is strictly negative, which is clearly not
  possible.
\end{Remark}

\section*{Acknowledgments}
We are very grateful to Pietro Caputo and Fabio Martinelli, who took
part in this project at an early stage, and to Julien Vovelle for
enlightening discussions on non-linear parabolic PDEs.  F.  T.  was
partially funded by Marie Curie IEF Action DMCP ``Dimers, Markov
chains and Critical Phenomena'', grant agreement n.  621894, by the
ANR-15-CE40-0020-03 Grant LSD and by the CNRS PICS grant ``Interfaces
al\'eatoires discr\`etes et dynamiques de Glauber''.  B. Laslier was
supported by the Engineering and Physical Sciences Research Council
under grant EP/103372X/1.

We wish to thank the two anonymous referees for a careful reading of the manuscript.


\begin{thebibliography}{99}

\bibitem{CMT} P. Caputo, F. Martinelli, F. L. Toninelli, \emph{Mixing times of monotone surfaces and SOS interfaces: a
mean curvature approach}, Comm. Math. Phys. {\bf 311} (2012), 157-189

\bibitem{CY} C. C. Chang and H.-T. Yau, \emph{Fluctuations of one dimensional Ginzburg-Landau models in 
nonequilibrium}, Commun. Math. Phys. {\bf 145} (1992), 
209-239.

\bibitem{CF} S. Chhita, P. L. Ferrari, \emph{A combinatorial identity for the speed of growth in an anisotropic KPZ
model},  Ann. Inst. Henri Poincar\'e D
{\bf 4}, Issue 4 (2017), 453–477

\bibitem{CT}  I. Corwin, F. L. Toninelli, \emph{Stationary measure of the driven two-dimensional q-Whittaker particle system on the torus}, Electronic Communications in Probability {\bf 21} (2016), paper no. 44 , 1-12. 

\bibitem{Fritz} J. Fritz, \emph{On the Hydrodynamic Limit of a Ginzburg Landau Lattice Model}, Prob. Th. Rel. Fields {\bf 81} (1989), 291-318.

  \bibitem{Funaki} T. Funaki, \emph{Stochastic interface models}. Lectures on probability theory and statistics, 103-274,
  Lecture Notes in Math., 1869, Springer, Berlin, 2005.

\bibitem{FunakiSpohn} T. Funaki, H. Spohn, 
\emph{Motion by Mean Curvature from the Ginzburg-Landau $\nabla\phi$
  Interface Model}, Comm. Math. Phys. {\bf 85} (1997), 1–36

\bibitem{Georgii} H.-O. Georgii, \emph{Gibbs measures and phase transitions}, Walter de Gruyter, 2011.

\bibitem{Kenyon}  R.  Kenyon,
\emph{Lectures  on  dimers},
Statistical  mechanics,  191-230,  IAS/Park  City  Math.  Ser.,
{\bf 16},
Amer. Math. Soc., Providence, RI, 2009

\bibitem{KOS} R. Kenyon, A. Okounkov, S. Sheffield, \emph{Dimers and amoebae}, Ann. Math. {\bf 163} (2006), 1019-1056.

\bibitem{KL}C. Kipnis, C. Landim, \emph{Scaling Limits of Interacting Particle Systems}, Springer, 1999.

\bibitem{LTcmp} B. Laslier, F. L. Toninelli, \emph{Lozenge tilings, Glauber dynamics and macroscopic shape}, Comm. Math.
Phys. {\bf 338} (2015), 1287-1326.

\bibitem{LThydro}B. Laslier, F. L. Toninelli, \emph{Hydrodynamic limit for a lozenge tiling Glauber dynamics},
 Annales  Henri Poincar\'e: Theor. Math. Phys. {\bf 18} (2017), 2007-2043.  

\bibitem{Liberman}G. M. Lieberman, \emph{Second order parabolic differential equations}, World Scientific, 1996

\bibitem{LRS}M. Luby, D. Randall, A. Sinclair, \emph{Markov Chain
    Algorithms for Planar Lattice Structures}, SIAM J. Comput.  {\bf
    31} (2001), 167-192.

\bibitem{Nishikawa}  T. Nishikawa,
\emph{Hydrodynamic limit for the Ginzburg-Landau $\nabla\phi$
interface model with boundary conditions}, Comm. Math. Phys.
{\bf 127}
(2003), 205-227.

\bibitem{Sheffield} S. Sheffield, \emph{Random surfaces}, Ast\'erisque (2005).

\bibitem{Spohn}H. Spohn, \emph{Large Scale Dynamics of Interacting Particles}, Springer, 1991.

\bibitem{SpohnJSP} H. Spohn, \emph{Interface  motion  in  models  with  stochastic  dynamics},  J.  Stat.  Phys.
{\bf 71} 
(1993), 1081-1132.
\bibitem{T2+1} F. L. Toninelli, \emph{A $(2+1)$-dimensional growth process with explicit stationary measure}, Ann. Probab. {\bf 45} (2017), 2899-2940.

\bibitem{Wilson} D. B. Wilson, \emph{Mixing times of lozenge tiling and card shuffling Markov chains},     Ann. Appl. Probab.
    {\bf 14} (2004), 274--325.

\end{thebibliography}
\end{document}